\documentclass[oneside,english]{amsart}
\usepackage[T1]{fontenc}
\usepackage[latin9]{inputenc}
\usepackage[active]{srcltx}
\usepackage{color}
\usepackage{babel}
\usepackage{enumitem}
\usepackage{amstext}
\usepackage{amsthm}
\usepackage{amssymb}
\PassOptionsToPackage{normalem}{ulem}
\usepackage{ulem}
\usepackage[unicode=true,pdfusetitle,
 bookmarks=true,bookmarksnumbered=false,bookmarksopen=false,
 breaklinks=false,pdfborder={0 0 1},backref=section,colorlinks=false]
 {hyperref}

\makeatletter

\providecommand{\tabularnewline}{\\}

\numberwithin{equation}{section}
\numberwithin{figure}{section}
\theoremstyle{plain}
\newtheorem{thm}{\protect\theoremname}[section]
\theoremstyle{definition}
\newtheorem{defn}[thm]{\protect\definitionname}
\theoremstyle{plain}
\newtheorem{fact}[thm]{\protect\factname}
\theoremstyle{plain}
\newtheorem{prop}[thm]{\protect\propositionname}
\theoremstyle{plain}
\newtheorem{lem}[thm]{\protect\lemmaname}
\theoremstyle{remark}
\newtheorem{rem}[thm]{\protect\remarkname}
\newlist{casenv}{enumerate}{4}
\setlist[casenv]{leftmargin=*,align=left,widest={iiii}}
\setlist[casenv,1]{label={{\itshape\ \casename} \arabic*.},ref=\arabic*}
\setlist[casenv,2]{label={{\itshape\ \casename} \roman*.},ref=\roman*}
\setlist[casenv,3]{label={{\itshape\ \casename\ \alph*.}},ref=\alph*}
\setlist[casenv,4]{label={{\itshape\ \casename} \arabic*.},ref=\arabic*}
\theoremstyle{plain}
\newtheorem{cor}[thm]{\protect\corollaryname}
\theoremstyle{lemma}
\newtheorem{mainlemma}[thm]{\protect\mainlemmaname}
\theoremstyle{remark}
\newtheorem{claim}[thm]{\protect\claimname}
\theoremstyle{plain}
\newtheorem{question}[thm]{\protect\questionname}

\usepackage{amsfonts}
\usepackage{nicefrac}
\usepackage{amscd}

\usepackage{a4wide}
\usepackage{url}
\usepackage{amssymb,graphicx,amsmath}

\newcommand{\trianglerightneq}{\mathrel{\ooalign{\raisebox{-0.5ex}{\reflectbox{\rotatebox{90}{$\nshortmid$}}}\cr$\triangleright$\cr}\mkern-3mu}}
\newcommand{\triangleleftneq}{\mathrel{\reflectbox{$\trianglerightneq$}}}

\makeatletter
\@namedef{subjclassname@2020}{%
  \textup{2020} Mathematics Subject Classification}
\makeatother

\AtBeginDocument{
  
}

\makeatother

\providecommand{\casename}{Case}
\providecommand{\claimname}{Claim}
\providecommand{\corollaryname}{Corollary}
\providecommand{\definitionname}{Definition}
\providecommand{\factname}{Fact}
\providecommand{\lemmaname}{Lemma}
\providecommand{\mainlemmaname}{Main Lemma}
\providecommand{\propositionname}{Proposition}
\providecommand{\questionname}{Question}
\providecommand{\remarkname}{Remark}
\providecommand{\theoremname}{Theorem}

\begin{document}
\global\long\def\code#1{\ulcorner#1\urcorner}%
\global\long\def\p{\mathbf{p}}%
\global\long\def\q{\mathbf{q}}%
\global\long\def\C{\mathfrak{C}}%
\global\long\def\SS{\mathcal{P}}%
 
\global\long\def\pr{\operatorname{pr}}%
\global\long\def\image{\operatorname{im}}%
\global\long\def\otp{\operatorname{otp}}%
\global\long\def\dec{\operatorname{dec}}%
\global\long\def\suc{\operatorname{suc}}%
\global\long\def\pre{\operatorname{pre}}%
\global\long\def\qe{\operatorname{qf}}%
 
\global\long\def\ind{\operatorname{ind}}%
\global\long\def\Nind{\operatorname{Nind}}%
\global\long\def\lev{\operatorname{lev}}%
\global\long\def\Suc{\operatorname{Suc}}%
\global\long\def\HNind{\operatorname{HNind}}%
\global\long\def\minb{{\lim}}%
\global\long\def\concat{\frown}%
\global\long\def\cl{\operatorname{cl}}%
\global\long\def\tp{\operatorname{tp}}%
\global\long\def\id{\operatorname{id}}%
\global\long\def\cons{\left(\star\right)}%
\global\long\def\qf{\operatorname{qf}}%
\global\long\def\ai{\operatorname{ai}}%
\global\long\def\dtp{\operatorname{dtp}}%
\global\long\def\acl{\operatorname{acl}}%
\global\long\def\nb{\operatorname{nb}}%
\global\long\def\limb{{\lim}}%
\global\long\def\leftexp#1#2{{\vphantom{#2}}^{#1}{#2}}%
\global\long\def\intr{\operatorname{interval}}%
\global\long\def\atom{\emph{at}}%
\global\long\def\I{\mathfrak{I}}%
\global\long\def\uf{\operatorname{uf}}%
\global\long\def\ded{\operatorname{ded}}%
\global\long\def\Ded{\operatorname{Ded}}%
\global\long\def\Df{\operatorname{Df}}%
\global\long\def\Th{\operatorname{Th}}%
\global\long\def\eq{\operatorname{eq}}%
\global\long\def\Aut{\operatorname{Aut}}%
\global\long\def\ac{ac}%
\global\long\def\DfOne{\operatorname{df}_{\operatorname{iso}}}%
\global\long\def\modp#1{\pmod#1}%
\global\long\def\sequence#1#2{\left\langle #1\,\middle|\,#2\right\rangle }%
\global\long\def\set#1#2{\left\{  #1\,\middle|\,#2\right\}  }%
\global\long\def\Diag{\operatorname{Diag}}%
\global\long\def\Nn{\mathbb{N}}%
\global\long\def\mathrela#1{\mathrel{#1}}%
\global\long\def\twiddle{\mathord{\sim}}%
\global\long\def\Sk{\operatorname{Sk}}%
\global\long\def\mathordi#1{\mathord{#1}}%
\global\long\def\Qq{\mathbb{Q}}%
\global\long\def\dense{\operatorname{dense}}%
\global\long\def\Rr{\mathbb{R}}%
 
\global\long\def\cof{\operatorname{cf}}%
\global\long\def\tr{\operatorname{tr}}%
\global\long\def\treeexp#1#2{#1^{\left\langle #2\right\rangle _{\tr}}}%
\global\long\def\x{\times}%
\global\long\def\forces{\Vdash}%
\global\long\def\Vv{\mathbb{V}}%
\global\long\def\Uu{\mathbb{U}}%
\global\long\def\tauname{\dot{\tau}}%
\global\long\def\ScottPsi{\Psi}%
\global\long\def\cont{2^{\aleph_{0}}}%
\global\long\def\MA#1{{MA}_{#1}}%
\global\long\def\rank#1#2{R_{#1}\left(#2\right)}%
\global\long\def\cal#1{\mathcal{#1}}%
\global\long\def\triangleleftnneq{\triangleleftneq}%

%% Independence 
\def\Ind#1#2{#1\setbox0=\hbox{$#1x$}\kern\wd0\hbox to 0pt{\hss$#1\mid$\hss} \lower.9\ht0\hbox to 0pt{\hss$#1\smile$\hss}\kern\wd0} 
\def\Notind#1#2{#1\setbox0=\hbox{$#1x$}\kern\wd0\hbox to 0pt{\mathchardef \nn="3236\hss$#1\nn$\kern1.4\wd0\hss}\hbox to 0pt{\hss$#1\mid$\hss}\lower.9\ht0 \hbox to 0pt{\hss$#1\smile$\hss}\kern\wd0} 
\def\nind{\mathop{\mathpalette\Notind{}}} 

\global\long\def\ind{\mathop{\mathpalette\Ind{}}}%
\global\long\def\Age{\operatorname{Age}}%
\global\long\def\lex{\operatorname{lex}}%
\global\long\def\len{\operatorname{len}}%

\global\long\def\dom{\operatorname{Dom}}%
\global\long\def\res{\operatorname{res}}%
\global\long\def\alg{\operatorname{alg}}%
\global\long\def\dcl{\operatorname{dcl}}%
 
\global\long\def\nind{\mathop{\mathpalette\Notind{}}}%
\global\long\def\average#1#2#3{Av_{#3}\left(#1/#2\right)}%
\global\long\def\Ff{\mathfrak{F}}%
\global\long\def\mx#1{Mx_{#1}}%
\global\long\def\maps{\mathfrak{L}}%

\global\long\def\Esat{E_{\mbox{sat}}}%
\global\long\def\Ebnf{E_{\mbox{rep}}}%
\global\long\def\Ecom{E_{\mbox{com}}}%
\global\long\def\BtypesA{S_{\Bb}^{x}\left(A\right)}%
\global\long\def\DenseTrees{T_{dt}}%

\global\long\def\init{\trianglelefteq}%
\global\long\def\fini{\trianglerighteq}%
\global\long\def\Bb{\cal B}%
\global\long\def\Lim{\operatorname{Lim}}%
\global\long\def\Succ{\operatorname{Succ}}%

\global\long\def\SquareClass{\cal M}%
\global\long\def\leqstar{\leq_{*}}%
\global\long\def\average#1#2#3{Av_{#3}\left(#1/#2\right)}%
\global\long\def\cut#1{\mathfrak{#1}}%
\global\long\def\NTPT{\text{NTP}_{2}}%
\global\long\def\Zz{\mathbb{Z}}%
\global\long\def\TPT{\text{TP}_{2}}%
\global\long\def\supp{\operatorname{supp}}%

\global\long\def\OurSequence{\mathcal{I}}%
\global\long\def\SUR{SU}%
\global\long\def\ShiftGraph#1#2{Sh_{#2}\left(#1\right)}%
\global\long\def\ShiftGraphHalf#1{\cal G_{#1}^{\frac{1}{2}}}%
\global\long\def\SymShiftGraph#1#2{Sh_{#2}^{sym}\left(#1\right)}%
\global\long\def\aa{\textsf{aa}}%
\global\long\def\Mm{\mathbb{M}}%
\global\long\def\stat{\textsf{stat}}%
\global\long\def\PC#1#2{PC\left(#1,#2\right)}%
\global\long\def\loc{\operatorname{loc}}%

\title{Exact saturation in pseudo-elementary classes for simple and stable
theories}
\author{Itay Kaplan, Nicholas Ramsey, and Saharon Shelah}
\thanks{The first-named author would like to thank the Israel Science Foundation
for partial support of this research (Grants no. 1533/14 and 1254/18). }
\thanks{The third-named author would like to thank the Israel Science Foundation
grant no: 1838/19 and the European Research Council grant 338821.
Paper no. 1197 in the third author\textquoteright s publication list.}
\subjclass[2020]{03C45, 03C55, 03C95}
\begin{abstract}
We use exact saturation to study the complexity of unstable theories, showing that a variant of this notion called PC-exact saturation meaningfully reflects combinatorial dividing lines.  We study PC-exact saturation for stable and simple theories. Among
other results, we show that PC-exact saturation characterizes the
stability cardinals of size at least continuum of a countable stable
theory and, additionally, that simple unstable theories have PC-exact
saturation at singular cardinals satisfying mild set-theoretic hypotheses.  This had previously been open even for the random graph. We characterize
supersimplicity of countable theories in terms of having PC-exact
saturation at singular cardinals of countable cofinality. We also
consider the local analogue of PC-exact saturation, showing that local
PC-exact saturation for singular cardinals of countable cofinality
characterizes supershort theories. 
\end{abstract}

\maketitle

\section{Introduction}

One of the major goals of model theory is to develop techniques to quantify and compare the complexity of mathematical theories.  This began with Morley's work on categoricity and the third-named author's work characterizing when models of a theory may be determined by an assignment of cardinal invariants.  This work spawned a rich structure theory for stable theories, which enabled many applications.  More recent work has pushed the boundaries of our understanding beyond stable theories, which has in turn required the development of new ways of measuring `complexity.'  In this paper, we introduce and develop the study of \emph{PC-exact saturation}, showing that this notion meaningfully measures how combinatorially complicated a theory can be.  In addition to giving an essentially complete description of how PC-exact saturation behaves in stable theories, we prove the existence of PC-exact saturated models of simple theories.  The constructions done here for simple theories are interesting in their own right\textemdash establishing techniques for model-construction which had been previously unknown even for the random graph\textemdash but, additionally, provide tantalizing new tools for possibly understanding further unstable classes beyond simplicity (e.g., NSOP$_{1}$ and NTP$_{1}$).  

Our work here builds upon two distinct lines of model-theoretic research.  A recurrent theme in model theory is the connection between the combinatorics
of definable sets--measured by dividing lines like stability, simplicity,
and NSOP$_{2}$--and the ability to construct saturated models of
a given first-order theory. Recently,
the notion of \emph{exact saturation }has proved to be especially
meaningful in relation to combinatorial dividing lines. A model is
said be be \emph{exactly} $\kappa$\emph{-saturated} if it is $\kappa$-saturated
but not $\kappa^{+}$-saturated. Saturation is a kind of completeness condition and exact saturation names an \emph{incompactness} phenomenon, in which completeness does not spill over.  The study of this notion was begun
in \cite{20-Kaplan2015} which showed, for singular cardinals, that, modulo some natural set-theoretic
hypotheses, among NIP theories, the existence of exactly saturated
models is characterized by the non-distality of the theory and, additionally, that exactly saturated models of simple theories exist. Later,
in \cite{kaplan2020criteria}, exact saturation motivated the discovery
of a new dividing line, namely the \emph{unshreddable theories}, containing
the simple and NIP theories and giving a general setting in which
exactly saturated models may be constructed. A second body of work in model theory concerns methods of directly comparing first order theories by comparing how difficult it is to saturate them.  The two main approaches to this are Keisler's order and the interpretability order; the former compares saturation of ultrapowers by the same ultrafilter, the latter describes when saturation of one theory may be transferred to another by means of a first-order theory that interprets both.  These `orders' are, in fact, pre-orders which divide theories into classes.  It was observed in \cite[Corollary 9.29]{MR3666452} that the existence of PC-exact saturated models, defined in the following paragraph, of a given theory depends only on its class in the interpretability order, and thus the tools developed here allow us to study in an \emph{absolute} way part of what these orders are studying in a \emph{relative} way.   

PC-exact saturation allows us to give many new characterizations
of dividing lines. We say that that an $L$-theory $T$ has \emph{PC-exact
saturation for $\kappa$} (where ``PC'' stands for \emph{pseudo-elementary
class}) if, for any $T_{1}\supseteq T$ with $\left|T_{1}\right|=\left|T\right|$,
there is a model $M\models T_{1}$ such that the reduct $M\upharpoonright L$
is $\kappa$-saturated but not $\kappa^{+}$-saturated. The PC version
turns out to be much more tightly connected to the complexity of the
theory $T$. For example, consider a two sorted $L$-structure $M=\left(X^{M},Y^{M}\right)$
where on one sort $X$, there is the structure of a countable model
of Peano arithmetic, and on the second sort $Y$ there is a countably
infinite set with no structure, with no relations between the two
sorts. From a model-theoretic point of view, $M$ is as complicated
as possible, because it interprets PA, but $\mathrm{Th}\left(M\right)$
has an exactly saturated model: interpreting $X$ as any $\kappa$-saturated
model of PA and $Y$ as a set with exactly $\kappa$ many elements
yields a $\kappa$-saturated model of $\mathrm{Th}\left(M\right)$ which is not $\kappa^{+}$-saturated.
Note, however, that $\mathrm{Th}\left(M\right)$ does \emph{not} have
PC-exact saturation for singular cardinals $\kappa$. In an expansion
of $M$ to $M'$, in a language containing a function symbol for a
bijection $f:X\to Y$, any model of $N\models\mathrm{Th}\left(M'\right)$
whose reduct to $L$ is $\kappa$-saturated will satisfy $\left|X^{N}\right|\geq\kappa^{+}$,
as PA has no exactly $\kappa$-saturated model (see \cite[Lemma 5.3 and the comment after]{kaplan2020criteria}),
hence $\left|Y^{N}\right|\geq\kappa^{+}$ and therefore $N\upharpoonright L$
is $\kappa^{+}$-saturated. It appears that, by looking at PC-exact
saturation, one obtains a condition that more
faithfully tracks the complicated combinatorics of the theory. %Indeed,
%in \cite[Corollary 9.29]{MR3666452} Malliaris and the third-named
%author observe that the question of whether or not a given theory
%has exact saturation for singular cardinals depends essentially only
%on its class in the interpretability order $\unlhd^{*}$, which is
%well-known as a measure of model-theoretic complexity.

One of the motivations of the work here was a question \cite[Question 9.31]{MR3666452}
of whether the random graph has PC-exact saturation. For an infinite
cardinal $\kappa$, one may easily construct an exactly $\kappa$-saturated
model of the random graph: in a $\kappa^{+}$-saturated random graph
$\mathcal{G}$, one can choose an empty (induced) subgraph $X\subseteq\mathcal{G}$
with $\left|X\right|=\kappa$ and set $\mathcal{G}'=\set{v\in\mathcal{G}}{\left|N_{\mathcal{G}}\left(v\right)\cap X\right|<\kappa}$,
where $N_{\mathcal{G}}\left(v\right)$ denotes the neighbors of the
vertex $v$. It is easily checked that, because $\mathcal{G}$ was
chosen to be $\kappa^{+}$-saturated, $\mathcal{G}'$ is a $\kappa$-saturated
random graph which omits the type $\set{R\left(x,v\right)}{v\in X}$
and hence is not $\kappa^{+}$-saturated. Given an expansion of $\mathcal{G}$
to a larger language $L$, it is less clear that one can arrange for
such a $\mathcal{G}'$ to be a reduct of a model of $\mathrm{Th}_{L}\left(\mathcal{G}\right)$.
We prove the existence of PC-exact saturated models of the random
graph by proving a much more general result, showing that, modulo
natural set-theoretic hypotheses, one may construct PC-exact saturated
models of simple theories and, along the way, we obtain many precise
equivalences between subclasses of the stable and simple theories. 

In Section \ref{sec:PC-Exact-saturation-for stable}, we begin our
study of PC-exact saturation by focusing on the stable theories. Here
we show that, for a countable stable theory $T$ and cardinal $\mu\geq2^{\aleph_{0}}$,
having PC-exact saturation at $\mu$ is equivalent to $T$ being stable
in $\mu$; this is Corollary \ref{cor:PC ES iff stable}. One direction
of this theorem is an easy consequence of the well-known fact that
stable theories have saturated models in the cardinals in which they
are stable and such models may be expanded to a model of any larger
theory (of the same size). The other direction involves the construction
of suitable expansions that allow one to find realizations of types
over sets of size $\mu$ from realizations of types of smaller
size. Although it does not appear in the statement, our proof relies
on a division into cases based on whether or not the given stable
theory has the finite cover property.

In Section \ref{sec:PC-exact-saturation-for simple }, we consider
simple unstable theories. In general, when $T$ is unstable then it
has PC-exact saturation at any large enough regular cardinal (Proposition
\ref{prop:unstable regular}), so we concentrate on singular cardinals
here. We show that simple unstable theories have PC-exact saturation
for singular cardinals $\kappa$, satisfying certain natural hypotheses.
In particular, this implies that the random graph has singular PC-exact
saturation, answering \cite[Question 9.31]{MR3666452}. In fact, we
prove that for a theory $T_{1}\supseteq T$ such that $\left|T_{1}\right|=\left|T\right|$,
there exists $M\models T_{1}$ such that $M\upharpoonright L$ is
$\kappa$-saturated but not even \emph{locally} $\kappa^{+}$-saturated,
that is, there is a partial type consisting of $\kappa$ many instances
of a single formula (or its negation) which is not realized; this
is Theorem \ref{thm:Main Thm simple theories}. In the supersimple
case we get a converse for singular cardinals of countable cofinality:
an unstable $L$-theory $T$ is supersimple iff $T$ has PC-exact
saturation at singular cardinals with countable cofinality (satisfying
mild hypotheses); this is Corollary \ref{cor:supersimple unstable char}.
Combined with the results described in the previous paragraph, we
get that a countable theory $T$ is supersimple iff for every $\kappa>2^{\aleph_{0}}$
of countable cofinality satisfying mild set-theoretic assumptions,
$T$ has PC-exact saturation at $\kappa$; this is Corollary \ref{cor:countable supersimpmle char}.

Finally, in Section \ref{sec:On-local-exact saturation and cofinality omega},
we elaborate on the case of \emph{local} PC-exact saturation at singular
cardinals. Here were are interested in determining, for a given $L$-theory
$T$ and singular cardinal $\kappa$ of countable cofinality, if,
for all $T_{1}\supseteq T$ with $\left|T_{1}\right|=\left|T\right|$,
there is $M\models T_{1}$ such that $M\upharpoonright L$ is locally
$\kappa$-saturated but not locally $\kappa^{+}$-saturated. We essentially
characterize when this takes place in terms of the notion of a supershort\emph{
}theory, a class of theories introduced by Casanovas and Wagner to
give a local analogue of supersimplicity \cite{MR1905165}. A theory
is called \emph{supershort} if every local type does not fork over
some finite set. The supershort theories properly contain the supersimple
theories. The main result of Section \ref{sec:On-local-exact saturation and cofinality omega}
shows that if $\kappa$ is a singular cardinal of countable cofinality,
satisfying natural hypotheses, then $T$ has local PC-exact saturation
for $\kappa$ if and only $T$ is supershort, giving the first ``outside''
characterization of this class of theories. 

We conclude in Section \ref{sec:Final-thoughts} with some questions
concerning possible extensions of the results of this paper.

\section{\label{sec:Preliminaries}Preliminaries}

\subsection{PC classes and exact saturation}

Here we give the basic definitions and facts about (PC-)exact saturation. 
\begin{defn}
\label{def:PC}Suppose that $T$ is a complete first order $L$-theory,
and let $L_{1}\supseteq L$. Suppose that $T_{1}$ is an $L_{1}$-theory.
Let $\PC{T_{1}}T$ be the class of models $M$ of $T$ which have
expansions $M_{1}$ to models of $T_{1}$. PC stands for ``pseudo
elementary class''. 
\end{defn}

\begin{defn}
\label{def:Exact saturation}Suppose $T$ is a first order theory
and $\kappa$ is a cardinal. 
\begin{enumerate}
\item Say that $T$ \emph{has exact saturation at $\kappa$} if $T$ has
a $\kappa$-saturated model $M$ which is not $\kappa^{+}$-saturated.
\item Say that a pseudo-elementary class $P=\PC{T_{1}}T$ has \emph{exact
saturation at a cardinal $\kappa$} if there is a $\kappa$-saturated
model $M\models T$ in $P$ which is not $\kappa^{+}$-saturated. 
\item Say that $T$ has \emph{PC-exact saturation} at a cardinal $\kappa$
if for every $T_{1}\supseteq T$ of cardinality $\left|T\right|$,
$\PC{T_{1}}T$ has exact saturation at $\kappa$. 
\end{enumerate}
\end{defn}

In general we expect that having exact saturation at $\kappa$ should
not depend on $\kappa$ (up to some set theoretic assumptions on $\kappa$).
For example, in \cite{20-Kaplan2015} the following facts were established:
\begin{fact}
\label{fact:exact saturation} Suppose $T$ is a first order theory.
\begin{enumerate}
\item \cite[Theorem 2.4]{20-Kaplan2015} If $T$ is stable then for all
$\kappa>\left|T\right|$, $T$ has exact saturation at $\kappa$.
\item \cite[Fact 2.5]{20-Kaplan2015} If $T$ is unstable then $T$ has
exact saturation at all regular $\kappa>\left|T\right|$.
\item \cite[Theorem 3.3]{20-Kaplan2015} Suppose that $T$ is simple, $\kappa$
is singular with $\left|T\right|<\mu=\cof\left(\kappa\right)$, $\kappa^{+}=2^{\kappa}$
and $\square_{\kappa}$ holds (see Definition \ref{def:square} below).
Then $T$ has exact saturation at $\kappa$. 
\item \cite[Theorem 4.10]{20-Kaplan2015} Suppose that $\kappa$ is a singular
cardinal such that $\kappa^{+}=2^{\kappa}$. An NIP theory $T$ with
$\left|T\right|<\kappa$ is distal (see e.g., \cite{MR3001548}) iff
it does not have exact saturation at $\kappa$. 
\end{enumerate}
\end{fact}

\subsection{Local (PC) exact saturation}
\begin{defn}
\label{def:local types}Work in some complete theory $T$. Given a
set $\Delta$ of formulas and a tuple of variables $x$, let $L_{x,\Delta}$
be the set of formulas of the form $\varphi\left(x\right)$ where
$\varphi$ is any formula in the Boolean algebra generated by $\Delta$
(where by $\varphi\left(x\right)$ we mean a substitution of the variables
in $\varphi$ by some variables from $x$). Similarly, given some
set $A$, $L_{x,\Delta}\left(A\right)$ is the set of formulas of
the form $\varphi\left(x,a\right)$ where $\varphi$ is any formula
from $L_{xy,\Delta}$ and $a$ is some tuple from $A$ (here $y$
is a countable sequence of variables). A\emph{ $\Delta$-type in variables
$x$ over a set $A$} is a maximal consistent collection of formulas
from $L_{x,\Delta}\left(A\right)$. A\emph{ local type} over $A$
is a $\Delta$-type for some finite set of formulas $\Delta$. The
space of all $\Delta$-types (in finitely many variables) over $A$
is denoted by $S_{\Delta}\left(A\right)$. (Sometimes this notion
is only defined when $\Delta$ is a partitioned set of formulas, $\Delta\left(x,y\right)$,
but here we allow all partitions.)

We say that a structure $M$ is \emph{$\kappa$-locally saturated}
if every local type (in finitely many variables) over a set $A\subseteq M$
of size $\left|A\right|<\kappa$ is realized. Say that $M$ is \emph{locally
saturated} if it is $\left|M\right|$-locally saturated. 
\end{defn}

We define a local analog to Definition \ref{def:Exact saturation}. 
\begin{defn}
\label{def:Exact saturation-local}Suppose $T$ is a first order theory
and $\kappa$ is a cardinal. 
\begin{enumerate}
\item Say that $T$ \emph{has local exact saturation at $\kappa$} if $T$
has a $\kappa$-locally saturated model $M$ which is not $\kappa^{+}$-locally
saturated.
\item Say that a PC-class $P=\PC{T_{1}}T$ has \emph{local exact saturation
at a cardinal $\kappa$} if there is a $\kappa$-locally saturated
model $M\models T$ in $P$ which is not $\kappa^{+}$-locally saturated. 
\item Say that $T$ has \emph{local PC-exact saturation} at a cardinal $\kappa$
if for every $T_{1}\supseteq T$ of cardinality $\left|T\right|$,
$\PC{T_{1}}T$ has local exact saturation at $\kappa$. 
\end{enumerate}
\end{defn}

\subsection{PC-exact saturation for unstable theories in regular cardinals}

For unstable theories and regular cardinals, the situation is as in
Fact \ref{fact:exact saturation} (2). Because of the following proposition,
when we discuss unstable theories, we will subsequently concentrate
on singular cardinals. 
\begin{prop}
\label{prop:unstable regular}If $T$ is not stable then $T$ has
PC-exact saturation at any regular $\kappa > |T|$. Moreover for any
$T_{1}\supseteq T$ of size $\left|T\right|$, there is a $\kappa$-saturated
model of $T_{1}$ whose reduct to the language $L$ of $T$ is not
$\kappa^{+}$-locally saturated. 
\end{prop}

\begin{proof}
The proof is almost exactly the same as the proof of \cite[Fact 2.5]{20-Kaplan2015}.

Let $M_{0}\models T_{1}$ be of size $\left|T\right|$. For $i\leq\kappa$,
define a continuous increasing sequence of models $M_{i}$ where $\left|M_{i+1}\right|=2^{\left|M_{i}\right|}$
and $M_{i+1}$ is $\left|M_{i}\right|^{+}$-saturated. Hence $M_{\kappa}$
is $\kappa$-saturated and $\left|M_{\kappa}\right|=\beth_{\kappa}\left(\left|T\right|\right)$.

As $T$ is unstable, and $M_{\kappa}$ is $\beth_{\kappa}\left(\left|T\right|\right)^{+}$-universal
$\left|S_{L}\left(M_{\kappa}\restriction L\right)\right|>\beth_{\kappa}\left(\left|T\right|\right)$
(for an explanation see the proof of \cite[Fact 2.5]{20-Kaplan2015}). 

However, as the number of $L$-types over $M_{\kappa}$ which are
invariant (i.e., which do not split) over $M_{i}$ is $\leq2^{2^{\left|M_{i}\right|}}\leq\beth_{\kappa}\left(\left|T\right|\right)$,
there is $p\left(x\right)\in S_{L}\left(M_{\kappa}\right)$ which
splits over every $M_{i}$. Hence for each $i<\kappa$, there is some
$L$-formula $\varphi_{i}\left(x,y\right)$ and some $a_{i},b_{i}\in M_{\kappa}$
such that $a_{i}\equiv_{M_{i}}b_{i}$ and $\varphi_{i}\left(x,a_{i}\right)\land\neg\varphi_{i}\left(x,b_{i}\right)\in p$.
As $\kappa>\left|T\right|$, there is a cofinal subset $E\subseteq\kappa$
such that for $i\in E$, $\varphi_{i}=\varphi$ is constant. Let $q\left(x\right)$
be $\set{\varphi\left(x,a_{i}\right)\land\neg\varphi\left(x,b_{i}\right)}{i\in E}$.
Then the local type $q$ is not realized in $M_{\kappa}$. 
\end{proof}

\section{\label{sec:PC-Exact-saturation-for stable}PC-Exact saturation for
stable theories}

The goal of this section is Theorem \ref{thm:PC exact saturation for stable theories}:
assuming that $T$ is a strictly stable countable complete theory
in the language $L$ and $\mu$ is a cardinal in which $T$ is \emph{not}
stable, we will find a countable $T_{1}\supseteq T$ in the language
$L_{1}\supseteq L$ such that if $M\models T_{1}$ and $M\upharpoonright L$
is $\aleph_{1}$-saturated and locally $\mu$-saturated, then $M\upharpoonright L$
is $\mu^{+}$-saturated.

In this section and in the sections that follow, we will often work with trees (usually $\omega^{<\omega}$).  We will write $\unlhd$ to denote the tree partial order and $\perp$ to denote the relation of incomparability, i.e. $\eta \perp \nu$ if and only if $\neg (\eta \unlhd \nu) \wedge \neg(\nu \unlhd \eta)$.  

\subsection{\label{subsec:Description-of-the expansion}Description of the expansion}

We will define our desired theory $T_{1}$ by choosing a certain model
of $T$, describing an expansion, and taking $T_{1}$ to be the theory
of this structure in the expanded language. We will assume that $L$
is disjoint from all symbols we are about to present.  As $T$ is
not superstable, we can use the following fact. 
\begin{fact}
\label{fact:not supersimple -> 2 incon}\cite[Proposition 3.5]{ArtemNick}
If $\kappa\left(T\right)>\aleph_{0}$ (namely, $T$ is not supersimple),
then there is a sequence of formulas $\sequence{\psi_{n}\left(x,y_{n}\right)}{n<\omega}$
(where $x$ is a single variable and the $y_{n}$'s are variables
of varying lengths) and a sequence $\sequence{a_{\eta}}{\eta\in\omega^{<\omega}}$
such that:
\begin{itemize}
\item $a_{\eta}$ is an $\left|y_{\left|\eta\right|}\right|$-tuple; for
$\sigma\in\omega^{\omega}$, $\set{\psi_{n}\left(x,a_{\sigma\restriction n}\right)}{n<\omega}$
is consistent; for every $\eta\in\omega^{n},\nu\in\omega^{m}$ such
that $\eta\perp\nu$, $\left\{ \psi_{n}\left(x,a_{\eta}\right),\psi_{m}\left(x,a_{\nu}\right)\right\} $
is inconsistent. 
\end{itemize}
\end{fact}

Fix a countable model $M_{0}\models T$ such that $\sequence{a_{\eta}}{\eta\in\omega^{<\omega}}$
is contained in $M_{0}$, and we will describe our expansion. Since
$M_{0}$ is countable we may assume that its universe is $\cal{\omega}\cup\omega^{<\omega}\cup\SS_{fin}\left(L\right)$
where $\SS_{fin}\left(L\right)$ is the set of all finite subsets
of formulas from $L$ in a fixed countable set of variables $\set{v_{i}}{i<\omega}$.
First expand $M_{0}$ by adding a predicate $\cal N$ for $\omega$
and adding $+,\cdot,<$ on $\cal N$ and a bijection $e:\mathcal{N}\to M_{0}$.
 We also have a predicate $\cal T$ for $\omega^{<\omega}$ on which
we add the order $\unlhd$. Add two functions $l:\mathcal{T}\to\mathcal{N}$
and $\mathrm{eval}:\mathcal{T}\times\mathcal{N}\to\mathcal{N}$ such
that $l$ is the length function and $\mathrm{eval}$ is the function
defined by $\mathrm{eval}\left(\eta,n\right)=\eta\left(n\right)$
for $n<l\left(\eta\right)$ and otherwise $\mathrm{eval}\left(\eta,n\right)=0$
(outside of their domain we define them in an arbitrary way). Note
that, in terms of this structure, if $\eta\in\mathcal{T}$ and $n\in\mathcal{N}$,
the concatenation $\eta\frown\left\langle n\right\rangle $ is defined
as the unique element $\nu$ of $\mathcal{T}$ of length $l\left(\eta\right)+1$
such that $\mathrm{eval}\left(\nu,l\left(\eta\right)\right)=n$ and
for $k<l\left(\eta\right)$, $\mathrm{eval}\left(\nu,k\right)=\mathrm{eval}\left(\eta,k\right)$.
We similarly define $\left\langle n\right\rangle \concat\eta$. For
notational simplicity, when $\eta\in\mathcal{T}$ and $k<l\left(\eta\right)$,
we will write $\eta\left(k\right)$ instead of $\mathrm{eval}\left(\eta,k\right)$.
Additionally, we will always refer to $\mathcal{N}$ for the natural
numbers predicate and use $\omega$ for the standard natural numbers.
For convenience we add a predicate $\mathcal{M}$ for the universe. 

We will write $n^{n}$ for the definable set of $\eta\in\mathcal{T}$
such that $l\left(\eta\right)=n$ and for all $k<n$, $\eta\left(k\right)<n$.
Let $i$ be a function with domain $\set{\left(\eta,n\right)}{\eta\in n^{n},n\in\mathcal{N}}$
and range $\mathcal{N}$ such that $i\left(-,n\right):n^{n}\to\mathcal{N}$
is an injection onto an initial segment of $\mathcal{N}$. 

We will add a predicate $\mathcal{L}$ for $\SS_{fin}\left(L\right)$,
together with $\subseteq$ giving containment and a truth predicate;
formally, the set of formulas $\cal L_{0}$ is identified with atoms
in the Boolean algebra $\cal L$ and the truth valuation is a function
$TV$ from $\cal L_{0}\times\cal T$ to $\left\{ 0,1\right\} $ such
that $TV\left(\left\{ \varphi\right\} ,\eta\right)=1$ iff $\varphi$
holds in $M_{0}$ with the assignment $v_{n}\mapsto e\left(\eta\left(n\right)\right)$
for $n<l\left(\eta\right)$ and $v_{n}\mapsto0$ otherwise.

 Add a function $d:\cal N\to\cal L_{0}$ mapping $n$ to $\left\{ \varphi_{n}\left(v_{0};v_{1},\ldots,v_{\left|y_{n}\right|}\right)\right\} $
(where the formulas $\varphi_{n}$ are as in Fact \ref{fact:not supersimple -> 2 incon}). Let $a:\mathcal{T}\to\mathcal{T}$
be a function such that if $\eta\in\omega^{<\omega}$ then $a\left(\eta\right)\in\omega^{\left|y_{l\left(\eta\right)}\right|}$
and there is some $c \in M_{0}$ such that $TV\left(d\left(i\right),\left\langle c\right\rangle \concat a\left(\eta|i\right)\right) = 1$ for all $i < l(\eta)$ (i.e. such that $\{\varphi_{i}\left(x;a\left(\eta|i\right)\right) : i<l\left(\eta\right)\}$ is consistent) and such that if $\eta,\nu\in\cal T$ are incomparable
then there is no $c \in M_{0}$ such that both $TV\left(d\left(l\left(\eta\right)\right),\left\langle c\right\rangle \concat a\left(\eta\right)\right) = 1$ and $TV\left(d\left(l\left(\nu\right)\right),\left\langle c\right\rangle \concat a\left(\nu\right)\right) = 1$ (i.e. $\{\varphi_{l(\eta)}(x;a(\eta)), \varphi_{l(\nu)}(x;a(\nu))\}$ is inconsistent). This is a direct translation of the properties of
the formulas $\varphi_{n}$ described above. 

We add a bijection $c:\mathcal{N}\to\mathcal{L}$ that associates
to each natural number a finite set of formulas. We add a predicate
$P\subseteq\mathcal{T}\times\mathcal{L}$ such that $\left(\eta,\Delta\right)\in P^{M_{0}}$
if and only if $\sequence{e\left(\eta\left(i\right)\right)}{i<l\left(\eta\right)}$
is a $\Delta$-indiscernible sequence that extends to a $\Delta$-indiscernible
sequence of countable length. Now we define a function $F:\mathcal{T}\times\mathcal{L}\times\mathcal{N}\to\mathcal{M}$
on $M_{0}$ so that, if $\left(\eta,\Delta\right)\in P^{M_{0}}$,
then $F^{M_{0}}\left(\eta,\Delta,-\right):\mathcal{N}\to M_{0}$ is
a $\Delta$-indiscernible sequence extending $\sequence{e\left(\eta\left(i\right)\right)}{i<l\left(\eta\right)}$.
Otherwise, $F^{M_{0}}\left(\eta,\Delta,m\right)$ is defined arbitrarily.
%We also define a function $G: \mathcal{T} \times \mathcal{L} \times \mathcal{N} \to \mathcal{T}$ such that if $\eta \in \mathcal{T}$ and $(e(\eta|i))_{i < l(\eta)}$ is a $\Delta$-$n$-indiscernible sequence and $(\eta,\Delta,n) \not\in P^{M_{0}}$, then $G(\eta,\Delta,n)$ is defined to be some $\nu \in \mathcal{T}$ such that $(e(\nu|i))_{i < l(\nu)}$ is a maximal $\Delta$-$n$-indiscernible sequence extending $\eta$; in other cases, it is defined arbitrarily.  

We let $M_{1}=\left(M_{0},\mathcal{N},\mathcal{T},\mathcal{L}\right)$
with this additional structure and constants for all elements (so
that models are elementary extensions), and we set $T_{1}=\mathrm{Th}\left(M_{1}\right)$
and $L_{1}=L\left(T_{1}\right)$. 

\subsection{Properties of $T_{1}$}
\begin{lem}
\label{lem:function coding} In $T_{1}$, there is a definable function
$H:\mathcal{N}\times\mathcal{N}\to\mathcal{N}$ such that if $M\models T_{1}$
and $M\upharpoonright L$ is $\aleph_{1}$-saturated, then for any
function $f:\omega\to\mathcal{N}^{M}$, there is some $m_{f}\in\cal N^{M}$
with the property that 
\[
H^{M}\left(m_{f},n\right)=f\left(n\right)
\]
for all $n\in\omega$. 
\end{lem}

\begin{proof}
Define $H:\mathcal{N}\times\mathcal{N}\to\mathcal{N}$ as follows: 
\begin{enumerate}
\item If there is $k\in\mathcal{N}$ and $\eta\in\mathcal{T}$ with $l\left(\eta\right)=n$
and $\models TV\left(d\left(n+1\right),\left\langle g\right\rangle \concat a\left(\eta\concat\left\langle k\right\rangle \right)\right)$,
we set $H\left(g,n\right)=k$ (note that the second parameter of $TV$
gets an element from $\cal T$ so $g$ should be in $\cal N$ for
this to be well-defined).
\item Else, we set $H\left(g,n\right)=0$. 
\end{enumerate}
Note that if there are $\eta,\eta'\in\mathcal{T}$ of length $n$
and $k'\in\mathcal{N}$ such that $\models TV\left(d\left(n+1\right),\left\langle g\right\rangle \concat a\left(\eta\concat\left\langle k\right\rangle \right)\right)\land TV\left(d\left(n+1\right),\left\langle g\right\rangle \concat a\left(\eta'\concat\left\langle k'\right\rangle \right)\right)$
then $\eta\frown\langle k\rangle$ and $\eta'\frown\langle k'\rangle$
must be comparable hence the same (since this is true in $M_{0}$).
It follows that $k=k'$, which shows that $H$ is well-defined. Also,
$H$ is clearly definable. 

Let $M\models T_{1}$ be a model such that $M\upharpoonright L$ is
$\aleph_{1}$-saturated. Let $f:\omega\to\mathcal{N}^{M}$ be arbitrary
and we will define a path $\langle\eta_{i}:i<\omega\rangle$ in $\mathcal{T}^{M}$
with $l\left(\eta_{i}\right)=i+1$ inductively, by setting $\eta_{0}=\left\langle f\left(0\right)\right\rangle $
and $\eta_{i+1}=\eta_{i}\concat\left\langle f\left(i+1\right)\right\rangle $
(by this we mean that $\mathrm{eval}\left(\eta_{0},0^{M}\right)=f\left(0\right)$,
etc.; this is possible to do in any model of $T_{1}$). By the choice
of $a$, we know that $\set{TV\left(d\left(i+1\right),\left\langle x\right\rangle \concat a\left(\eta_{i}\right)\right)}{i<\omega}$
 is a partial type which in turn implies that $\set{\varphi_{i+1}\left(x,e\left(a\left(\eta_{i}\right)\left(0\right)\right),\ldots,e\left(a\left(\eta_{i}\right)\left(\left|y_{i+1}\right|-1\right)\right)\right)}{i<\omega}$
is a partial type and, by $\aleph_{1}$-saturation, this is realized
by some $m_{f}'\in M$. Unraveling definitions, we have $H^{M}\left(e^{-1}\left(m_{f}'\right),n\right)=f\left(n\right)$
for all $n\in\omega$, so let $m_{f}=e^{-1}\left(m_{f}'\right)$. 
\end{proof}
%\begin{lem} \label{easy lemma}
%Suppose $M \models T_{1}$, $k \in \mathcal{N}^{M}$, and $m \in \mathcal{M}^{M}$.  Then there is $\eta \in (k^{k})^{M}$ such that 
%$$
%\eta(i) = \min\{H(i,m), k-1\}
%$$
%for all $i < k$.  
%\end{lem}
%\begin{proof}
%Recall that $(k^{k})^{M}$ denotes $M$-definable set $\{\eta \in \mathcal{T}^{M} : l(\eta) = k \text{ and } \mathrm{ran}(\eta) \subseteq [0,k)\}$.  Note if $k < \omega$, and $m \in \mathcal{M}^{M_{0}}$, then there is an element $\eta \in k^{k}$ such that $\min\{H^{M_{0}}(i,m),k-1\} = \eta(i)$ for all $i < k$, and hence 
%$$
%M_{0} \models (\forall k \in \mathcal{N})(\forall m \in \mathcal{M})(\exists \eta \in \mathcal{T})\left[(\forall i < k)[\eta(i) = \min\{H(i,m),k-1\}]\right],
%$$
%and hence this sentence is contained in $T_{1}$, which proves the lemma.  
%\end{proof}
%\begin{lem} \label{easy lemma 2}
%Suppose $M \models T_{1}$, $\eta \in \mathcal{T}^{M}$, and $l(\eta) = n \in \mathcal{N}^{M}$ (possibly nonstandard).  Then given any $N < \omega$ with $N \leq n$, if $v \subseteq \mathrm{ran}(\eta)$ has size $\leq N$, then there is $\xi \in \mathcal{T}^{M}$ with $l(\xi) = N$ and $\mathrm{ran}(\xi) = v$.  
%\end{lem}
%\begin{proof}
%Note that $M_{0}$ satisfies the following sentence:
%$$
%(\forall n > N)(\forall \eta \in \mathcal{T})\left[ l(\eta) = n \to (\forall x_{0}, \ldots, x_{N-1} \in \mathrm{ran}(\eta))(\exists \xi \in \mathcal{T})[l(\xi) = N \wedge \bigwedge_{i < N}\xi(i) = x_{i}]\right],
%$$
%which, therefore, belongs to $T_{1}$.  Since $M \models T_{1}$, the claim follows.  
%\end{proof}
\begin{lem}
\label{fcp lemma} Suppose that $T$ has the fcp, $M\models T_{1}$
and $M\upharpoonright L$ is locally $\mu$-saturated. Then for any
non-standard $n\in\mathcal{N}^{M}$, we have $\left|\left[0,n\right)\right|\geq\mu$. 
\end{lem}

\begin{proof}
As $T$ has the finite cover property, there is a formula $\varphi\left(x,y\right)$
with $x$ a singleton such that in any model of $T$, for all $k<\omega$,
there are some $k'>k$ and $\sequence{a_{i}}{i\leq k'}$ such that
$\set{\varphi\left(x,a_{i}\right)}{i<k'}$ is inconsistent yet $k$-consistent.
For simplicity assume that $y$ is a singleton (otherwise replace
$\eta$ by a tuple of $\eta$'s in what follows). In particular, for
all $k<\omega$, $M_{1}$ satisfies $\chi\left(k\right)$ which asserts
that there exists $k'>k$ such that $k'<n$ and $\eta\in\mathcal{T}$
with $l\left(\eta\right)=k'$ such that the following two conditions are satisfied: 
\begin{enumerate}[label=(\Alph*)]
\item $\neg\exists x\in\mathcal{M}\left(\bigwedge_{i<k'}\varphi\left(x,e\left(\eta\left(i\right)\right)\right)\right)$. 
\item If $\nu\in\mathcal{T}$ satisfies $l\left(\nu\right)\leq k$ and $\mathrm{ran}\left(\nu\right)\subseteq\mathrm{ran}\left(\eta\right)$,
then 
\[
\exists x\in\mathcal{M}\bigwedge_{i<k}\varphi\left(x,e\left(\nu\left(i\right)\right)\right).
\]
\end{enumerate}
Note that this is expressible in $L_{1}$. By overspill, there is
$k\in\mathcal{N}^{M}$ nonstandard such that $\chi\left(k\right)$
holds witnessed by some $k'<n$. Without loss of generality, we may
assume $k'=n$, since it suffices to show $\left[0,k'\right)$ has
size $\mu$. Consequently, in $M$ there is an $\eta\in\mathcal{T}^{M}$
with $l\left(\eta\right)=n$ witnessing the above formula.

Note that, for each $N<\omega$, we have the following sentence is
satisfied in $M_{0}$: 
\[
\forall m>N\forall\eta\in\mathcal{T}\left[l\left(\eta\right)=m\to\left(\forall x_{0},\ldots,x_{N-1}\in\mathrm{ran}\left(\eta\right)\right)\left(\exists\xi\in\mathcal{T}\right)\left[\bigwedge_{i<N}\xi\left(i\right)=x_{i}\right]\right].
\]
It follows that $M$ satisfies this sentence and so, by Condition (B)
and compactness, we have that $\set{\varphi\left(x,e\left(\eta\left(i\right)\right)\right)}{i\in\left[0,n\right)}$
is consistent.

However, by the Condition (A), $\set{\varphi\left(x,e\left(\eta\left(i\right)\right)\right)}{i\in\left[0,n\right)}$
is not realized and therefore has size $\geq\mu$ by the local $\mu$-saturation
of $M\upharpoonright L$. 
\end{proof}
\begin{lem}
\label{tree lemma} Suppose $M\models T_{1}$ is given such that $M\upharpoonright L$
is $\aleph_{1}$-saturated. 
\begin{enumerate}
\item If $\left|M\right|\geq\mu$, then $\left|M\right|\geq\mu^{\aleph_{0}}$. 
\item If, moreover, $M$ has the property that $\left[0,n\right)\geq\mu$
for all nonstandard $n\in\mathcal{N}^{M}$, then for any nonstandard
$n\in\mathcal{N}^{M}$, $\left[0,n\right)\geq\mu^{\aleph_{0}}$. 
\end{enumerate}
\end{lem}

\begin{proof}
(1) Given any $f:\omega\to\mathcal{N}^{M}$, as $M\upharpoonright L$
is $\aleph_{1}$-saturated, we know, by Lemma \ref{lem:function coding},
there is some $m_{f}\in\cal N^{M}$ such that 
\[
H\left(m_{f},i\right)=f\left(i\right),
\]
for all $i<\omega$. As $e^{M}:\cal N^{M}\to\cal M^{M}$ is a bijection,
$\left|\cal M^{M}\right|=\left|\mathcal{N}^{M}\right|$ so, as there
are $\left|\mathcal{N}^{M}\right|^{\aleph_{0}}\geq\mu^{\aleph_{0}}$
functions $f:\omega\to\mathcal{N}^{M}$ and the map $f\mapsto m_{f}$
is clearly injective, we have $\left|M\right|\geq\mu^{\aleph_{0}}$,
which proves (1).

Now we prove (2). Recall that in $M_{0}\models T_{1}$, we defined
the function $i$ such that, for all $n$, $i\left(-,n\right):n^{n}\to\mathcal{N}$
is an injection onto an initial segment of $\mathcal{N}$.

Let $log:\mathcal{N}\to\mathcal{N}$ be the (definable) function given
by 
\[
log\left(n\right)=\max\set k{\mathrm{ran}\left(i\left(-,k\right)\right)\subseteq\left[0,n\right)}.
\]
On the standard natural numbers, $log\left(n\right)$ is the largest
$k$ such that $k^{k}\leq n$.

Given $n$ nonstandard, let $k=log^{M}\left(n\right)$. Because $i\left(-,k\right)$
gives an injection of $\left(k^{k}\right)^{M}$ into $\left[0,n\right)$,
we are reduced to showing that $\left|\left(k^{k}\right)^{M}\right|\geq\mu^{\aleph_{0}}$.
Note that $k$ is also nonstandard, so by hypothesis, $\left|\left[0,k\right)\right|\geq\mu$.
Let $f:\omega\to\left[0,k\right)$ be an arbitrary function. As in
(1), by Lemma \ref{lem:function coding}, there is some $m_{f}\in\mathcal{N}^{M}$
such that $H\left(m_{f},i\right)=f\left(i\right)$ for all $i<\omega$.

Note that if $j<\omega$, and $m\in\mathcal{N}^{M_{0}}$, then there
is an element $\eta\in j^{j}$ such that $\min\left\{ H^{M_{0}}\left(m,i\right),j-1\right\} =\eta\left(i\right)$
for all $i<j$, and hence 
\[
M_{0}\models\forall j\in\mathcal{N}\forall m\in\mathcal{N}\exists\eta\in j^{j}\left[\left(\forall i<j\right)\left[\eta\left(i\right)=\min\left\{ H\left(m,i\right),j-1\right\} \right]\right],
\]
and hence this sentence is contained in $T_{1}$. Applying this with
$j=k$ and $m=m_{f}$, we know there is $\tilde{\eta}_{f}\in\left(k^{k}\right)^{M}$
such that 
\[
\tilde{\eta}_{f}\left(i\right)=H\left(m_{f},i\right),
\]
for all $i<\omega$ (using that $H\left(m_{f},i\right)=f\left(i\right)<k$
for all $i<\omega$). Moreover, because $\tilde{\eta}_{f}\left(i\right)=f\left(i\right)$
for all standard $i$, we have $f\neq f'$ implies $\tilde{\eta}_{f}\neq\tilde{\eta}_{f'}$.
This shows $\left|\left(k^{k}\right)^{M}\right|\geq\mu^{\aleph_{0}}$,
which completes the proof. 
\end{proof}
\begin{fact}
\cite[Theorem II.4.6]{shelah1990classification} \label{fcp facts}
Suppose $T$ is nfcp, $\Delta$ is a finite set of formulas. There
is an $n\left(\Delta\right)<\omega$ such that, if $A$ is a set of
parameters and $\set{a_{\gamma}}{\gamma<\alpha}$ is a $\Delta$-indiscernible
set over $A$ and $n\left(\Delta\right)\leq\alpha<\beta$, then there
exist $a_{\gamma}$ for $\alpha\leq\gamma<\beta$ such that $\set{a_{\gamma}}{\gamma<\beta}$
is a $\Delta$-indiscernible set over $A$. 
\end{fact}

\begin{rem}
\label{rem:change of statement}In \cite[Theorem II.4.6]{shelah1990classification},
the result is a bit more refined since it applies to $\Delta$-$n$-indiscernible
sequences and gives one $n\left(\Delta\right)$ which works for all
$n$ (note that being $\Delta$-indiscernible is equivalent to being
$\Delta$-$n$-indiscernible for some large enough $n$), but this
is not needed here.
\end{rem}

\begin{rem}
\label{rem:witness for fcp in ground model}If, in the context of
the previous fact, $M$ is a model and we are given that $A$ is finite
and $\alpha<\beta\leq\omega$ and $Aa_{<\alpha}\subseteq M$, then
we can choose $a_{\gamma}\in M$ for $\alpha\leq\gamma<\beta$ (if
$\beta=\omega$ we extend one by one). Likewise, if $Aa_{<\alpha}\subseteq M$
and $M$ is $\left(\left|A\right|+\left|\beta\right|\right)^{+}$-saturated,
we may choose $a_{\gamma}\in M$ for $\alpha\leq\gamma<\beta$. 
\end{rem}

\begin{defn}
The infinite indiscernible sequences $I_{1}$ and $I_{2}$ are \emph{equivalent}
if there is an infinite indiscernible sequence $J$ such that $I_{1}\frown J$
and $I_{2}\frown J$ are both indiscernible. If $M$ is a model and
$I$ is an infinite indiscernible sequence contained in $M$, we define
$\mathrm{dim}\left(I,M\right)$ by 
\[
\mathrm{dim}\left(I,M\right)=\min\set{\left|J\right|}{J\text{ maximal indiscernible in }M\text{ equivalent to }I}.
\]
If $\mathrm{dim}\left(I,M\right)=\left|J\right|$ for all maximal
indiscernible sequences $J$ in $M$ equivalent to $I$, then we say that
the dimension is \emph{true}. 
\end{defn}

\begin{fact}
\label{extension fact} Suppose $T$ is a countable stable theory
and $M\models T$.
\begin{enumerate}
\item \cite[Theorem III.3.9]{shelah1990classification} For any infinite
indiscernible $I$ contained in $M$, the dimension $\mathrm{dim}\left(I,M\right)$
is true (there it is stated with the added condition that $dim\left(I,M\right)\geq\kappa\left(T\right)$
but this condition is not necessary when $T$ is countable as follows
from the proof there; however since we will later assume $\aleph_{1}$-saturation
we can just apply the reference as is). 
\item \cite[Theorem III.3.10]{shelah1990classification} Assume that $M$
is $\aleph_{1}$-saturated. For an infinite cardinal $\lambda$, $M$
is $\lambda$-saturated if and only if for every infinite indiscernible
sequence $I$ in $M$ of single elements, we have $\mathrm{dim}\left(I,M\right)\geq\lambda$.
(In the reference there is no restriction on the length of the tuples,
but it is enough to consider sequences of elements as can be seen
from the proof there and the fact that saturation is implied by realizing
1-types.)\footnote{If the reader is not satisfied with this, they can alter the language
$L_{1}$ by adding predicates $P_{n}$ and functions $F_{n}$ to deal
with sequences of length $n$, for all $n<\omega$. }
\end{enumerate}
\end{fact}

\begin{thm}
\label{thm:PC exact saturation for stable theories}Suppose $T$ is
a strictly stable countable theory and $\mu$ is a cardinal such that
$T$ is not stable in $\mu$. Then there is a countable theory $T_{1}\supseteq T$
with the property that for all $M\models T_{1}$, if $M\upharpoonright L$
is $\aleph_{1}$-saturated and locally $\mu$-saturated, then $M\upharpoonright L$
is $\mu^{+}$-saturated. 
\end{thm}

\begin{proof}
As $T$ is stable in every cardinal $\kappa$ satisfying $\kappa^{\left|T\right|}=\kappa$,
we have $\mu^{\aleph_{0}}>\mu$. Let $T_{1}$ be the theory described
above. Let $M\models T_{1}$ be arbitrary with the property that $M\upharpoonright L$
is $\aleph_{1}$-saturated and locally $\mu$-saturated. By Fact \ref{extension fact}(1),
the dimension of any infinite indiscernible sequence in $M$ is true.
Since every infinite indiscernible sequence is equivalent to all
of its infinite initial segments, we know that, by Fact \ref{extension fact}(2),
to show $M\upharpoonright L$ is $\mu^{+}$-saturated, it suffices
to show that every length $\omega$ indiscernible sequence (with respect
to the language $L$) contained in $M$ extends to one that has length
$\geq\mu^{+}$. Note that $M\succ M_{1}$ so in particular contains
all standard numbers and finite sets of formulas. 

Fix an arbitrary indiscernible sequence $I=\sequence{a_{i}}{i<\omega}$
of elements in $M$ and we will find an extension of length $\geq\mu^{+}$.
For each $i<\omega$, there is some $\eta_{i}\in\mathcal{T}^{M}$
with $l\left(\eta_{i}\right)=i$ and $\sequence{e\left(\eta_{i}\left(j\right)\right)}{j<i}=\sequence{a_{j}}{j<i}$.
We fix also an increasing sequence of finite sets of formulas $\Delta_{n}$
such that $L=\bigcup_{n<\omega}\Delta_{n}$.
\begin{casenv}
\item [Case 1.]For all $n<\omega$, $M\models P\left(\eta_{n},\Delta_{n}\right)$.
\end{casenv}
As $M\restriction L$ is, in particular $\aleph_{1}$-saturated, we
may apply Lemma \ref{lem:function coding} to find $m_{1},m_{2}\in\cal N^{M}$
such that $H\left(m_{1},n\right)=e^{-1}\left(a_{n}\right)$ and $H\left(m_{2},n\right)=c^{-1}\left(\Delta_{n}\right)$
for all $n<\omega$. For $n\in\mathcal{N}^{M}$, let $\nu\left(n\right)$
be the element of $\mathcal{T}$ such that $l\left(\nu\left(n\right)\right)=n$
and $\mathrm{eval}\left(\nu\left(n\right),i\right)=H\left(m_{1},i\right)$
for all $i<n$. Notice that we have that the function $n\mapsto\nu\left(n\right)$
is definable in $M$ and $\nu\left(n\right)=\eta_{n}$ for all (standard)
$n<\omega$. Next, let $\Delta\left(n\right)=c\left(H\left(m_{2},n\right)\right)$
for all $n\in\mathcal{N}^{M}$. Likewise, we have that the function
$n\mapsto\Delta\left(n\right)$ is $M$-definable and $\Delta\left(n\right)=\Delta_{n}$
for all $n<\omega$.

By our assumptions, we have that for all $n<\omega$, 
\[
M\models\forall k\leq n\left(\Delta\left(k\right)\subseteq\Delta\left(n\right)\wedge\nu\left(k\right)\unlhd\nu\left(n\right)\right)\wedge P\left(\nu\left(n\right),\Delta\left(n\right)\right),
\]
and, hence, by overspill (which we may apply since nonstandard elements
exists by $\aleph_{1}$-saturation), there is some nonstandard $n_{*}\in\mathcal{N}^{M}$
such that 
\[
M\models\forall k\leq n_{*}\left(\Delta\left(k\right)\subseteq\Delta\left(n_{*}\right)\wedge\nu\left(k\right)\unlhd\nu\left(n_{*}\right)\right)\wedge P\left(\nu\left(n_{*}\right),\Delta\left(n_{*}\right)\right).
\]
Since $n_{*}$ is nonstandard, we know that $\Delta\left(n_{*}\right)$
contains all standard formulas of $L$ and $\nu\left(n_{*}\right)$
extends $\nu\left(k\right)$ for all $k<\omega$. Since, additionally,
$M\models P\left(\nu\left(n_{*}\right),\Delta\left(n_{*}\right)\right)$,
it follows that $F^{M}\left(\nu\left(n_{*}\right),\Delta\left(n_{*}\right),-\right):\mathcal{N}^{M}\to M$
enumerates an $L$-indiscernible sequence extending $I$. (Note that
this was a property of $F^{M_{0}}$ and using the truth predicate
the fact that this is an indiscernible sequence is expressible in
$L_{1}$ so this is also true in $M$; this also uses the fact that
for standard formulas, the truth predicate gives the correct answer.)

The local $\mu$-saturation of $M\restriction L$ implies $\left|M\right|\geq\mu$,
so, by Lemma \ref{tree lemma}(1), we have $\left|\mathcal{N}^{M}\right|=\left|M\right|\geq\mu^{\aleph_{0}}\geq\mu^{+}$,
we have shown that $I$ extends to an indiscernible sequence of length
$\mu^{+}$.
\begin{casenv}
\item [Case 2.]There is an $N<\omega$ such that $M\models\neg P\left(\eta_{N},\Delta_{N}\right)$.
\end{casenv}
Note that it follows that for all $\omega>r\geq N$, $M\models\neg P\left(\eta_{r},\Delta_{N}\right)$
(since $\eta_{r}$ extends $\eta_{N}$ and this implication is true
in $M_{1}$). If $T$ was nfcp, then in particular this is true for
any $r\geq N,n\left(\Delta_{N}\right)$, where $n\left(\Delta_{N}\right)$
is from Fact \ref{fcp facts}. Thus by elementarity, 
\[
M_{1}\models\exists\eta\in\mathcal{T}\left[l\left(\eta\right)=r\wedge\neg P\left(\eta,\Delta_{N}\right)\wedge\sequence{e\left(\eta\left(i\right)\right)}{i<r}\text{ is }\Delta_{N}\text{-indiscernible}\right]
\]
hence $T$ has the finite cover property by choice of $r$ and Remark
\ref{rem:witness for fcp in ground model}. %Moreover, our assumption entails that for all $n < \omega$ with $n > N$, we have $M \models \neg P(\eta_{n},\Delta_{n},n)$ as well.  
However, because $I$ is $L$-indiscernible, we have $\nu\left(n\right)$
is $\Delta_{n}$-indiscernible for all $n<\omega$. Therefore, we
have that for all $n<\omega$, $M$ satisfies the conjunction of the
following sentences (which can be expressed using the truth predicate): 

$\forall k\leq n\left(\left(\Delta\left(k\right)\right)\subseteq\Delta\left(n\right)\wedge\nu\left(k\right)\unlhd\nu\left(n\right)\right)$. 

$\forall k\leq n\left[\sequence{e\left(\nu\left(n\right)\left(i\right)\right)}{i<n}\text{is }\Delta\left(n\right)\text{-indiscernible}\right]$.
%\item $(n \geq N \to \neg P(\nu(n),\Delta(n),n))$.

Thus, by overspill, there is some nonstandard $n_{*}\in\mathcal{N}^{M}$
such that $\sequence{e\left(\nu\left(n_{*}\right)\left(i\right)\right)}{i<n_{*}}$
is $L$-indiscernible and extends $I$. By Lemma \ref{fcp lemma}
and Lemma \ref{tree lemma}(2), we know $\left|\left[0,n_{*}\right)\right|\geq\mu^{\aleph_{0}}>\mu$,
so we have shown $I$ extends to an indiscernible sequence of length
$\geq\mu^{+}$.

This completes the proof. 
\end{proof}
\begin{cor}
\label{cor:PC ES iff stable}For a countable stable theory $T$ and
$\mu\geq2^{\aleph_{0}}$, the following are equivalent:
\begin{enumerate}
\item $T$ has PC-exact saturation in $\mu$. 
\item $T$ is stable in $\mu$. 
\end{enumerate}
\end{cor}

\begin{proof}
(1) implies (2): if $T$ is superstable then (2) holds automatically.
Otherwise, we are done by Theorem \ref{thm:PC exact saturation for stable theories}. 

(2) implies (1): if $T$ is stable in $\mu$, then by \cite[Chapter VIII, Theorem 4.7]{shelah1990classification}
$T$ has a saturated model $M$ of size $\mu$. Since saturated models
are resplendent \cite[Theorem 9.17]{MR1757487} we can expand $M$
to a model $M'$ of $T_{1}$. 
\end{proof}

\section{\label{sec:PC-exact-saturation-for simple }PC-exact saturation for
simple and supersimple theories}

\subsection{Simple theories}

This section is devoted to the proof of the following theorem, which
in particular answers \cite[Question 9.31]{MR3666452} (about the
random graph). 
\begin{thm}
\label{thm:Main Thm simple theories} Assume that $T$ is a complete
simple unstable $L$-theory, $T_{1}\supseteq T$ is a theory in $L_{1}\supseteq L$,
$\left|T_{1}\right|\leq\left|T\right|$. Also, assume that $\kappa$
is a singular cardinal such that $\kappa\left(T\right)\leq\mu=\cof\left(\kappa\right)$,
$\left|T\right|<\kappa$, $\kappa^{+}=2^{\kappa}$ and $\square_{\kappa}$
holds (see Definition \ref{def:square}). Then $\PC{T_{1}}T$ has
exact saturation at $\kappa$. Moreover, there is a model $M\models T_{1}$
whose reduct to $L$ is $\kappa$-saturated but not $\kappa^{+}$-locally
saturated. 
\end{thm}

The proof is somewhat similar to the proof of the parallel theorem
from \cite{20-Kaplan2015} (Fact \ref{fact:exact saturation} (3)),
with some important differences. The class $\cal M$ from there is
similar to the class $\cal C$ here, and the overall structure of
the proof is similar, but the proof of the main lemma (Lemma \ref{lem:Main lemma simple}
below) is quite different. 

We may assume that $T_{1}$ has built-in Skolem functions. Since
$T$ is unstable and simple it has the independence property, as witnessed
by some formula $\varphi\left(x,y\right)$ from $L$. Suppose that
$\kappa=\sum_{i<\mu}\lambda_{i}$ where the sequence $\sequence{\lambda_{i}}{i<\mu}$
is continuous, increasing, and each $\lambda_{i}$ is regular for
$i<\mu$ a successor. Also, assume that $\left|T\right|\leq\lambda_{0}$
(here we use the assumption that $\left|T\right|<\kappa$).  We work
in a monster model $\C_{1}$ of $T_{1}$, and denote its $L$-reduct
by $\C$. Let $I=\sequence{a_{\alpha}}{\alpha<\kappa}$ be an $L_{1}$-indiscernible
sequence of $y$-tuples of order type $\kappa$, which witnesses that
$\varphi$ has IP. For $i<\mu$, let $I_{i}=\sequence{a_{\alpha}}{\alpha<\lambda_{i}}$.
Also, for $\alpha<\kappa$, let $\bar{a}_{\alpha}$ be the sequence
$\sequence{a_{\omega\alpha+k}}{k<\omega}$. 
\begin{defn}
\label{def:the class C}Let $\cal C$ be the class of sequences $\sequence{A_{i}}{i<\kappa}$
such that:
\begin{enumerate}
\item For all $i<\kappa$, $A_{i}\prec\C_{1}$; $\left|A_{i}\right|\leq\lambda_{i}$;
$I_{i}\subseteq A_{i}$. 
\item The sequence $\sequence{A_{i}}{i<\kappa}$ is increasing continuous.
\item For all $i<\kappa$ and every finite tuple $c\in A_{i+1}$, there is
a club $E\subseteq\lambda_{i+1}$ such that for all $\alpha\in E$,
$\bar{a}_{\alpha}$ is $L_{1}$-indiscernible over $c$, where $\overline{a}_{\alpha} = (a_{\omega \alpha + k})_{k < \omega}$.
\end{enumerate}
For $\bar{A},\bar{B}\in\cal C$, write $\bar{A}\leq\bar{B}$ for:
for every $i<\mu$, $A_{i}\subseteq B_{i}$. 
\end{defn}

For example, letting $A_{i}=\Sk\left(I_{i}\right)$ (the Skolem hulls
of $I_{i}$) for $i<\mu$, $\bar{A}=\sequence{A_{i}}{i<\mu}\in\cal C$.
Then $\bar{A}$ is $\leq$-minimal. 
\begin{mainlemma}
\label{lem:Main lemma simple} Suppose that $\bar{A}\in\cal C$ and
let $A=\bigcup\set{A_{i}}{i<\mu}$. Suppose that $C\subseteq A$ has
size $<\kappa$ and $p\left(x\right)\in S_{L}\left(C\right)$. Then
there is $\bar{B}\in\cal C$ such that $\bar{A}\leq\bar{B}$ and $B=\bigcup\set{B_{i}}{i<\mu}$
realizes $p$. 
\end{mainlemma}

\begin{proof}
We want to make the following assumptions first. By maybe increasing
$C$ by a set of size $\leq\left|T\right|+\left|C\right|+\mu$ we
may assume that:
\begin{itemize}
\item [$(\star)$] $C\ind_{A_{i}\cap C}A_{i}$ and $A_{i}\cap C\prec\C_{1}$
for all $i<\mu$. ($\ind$ denotes non-forking independence.)
\end{itemize}
(This is straightforward, but for a proof see the beginning of the
proof in \cite[Main Lemma 3.11]{20-Kaplan2015}.)

Without loss of generality assume that  $p$ does not fork over $A_{*}=C\cap A_{0}$
(this uses the assumption that $\cof\left(\kappa\right)=\mu\geq\kappa\left(T\right)$).
(Let $i_{*}<\mu$ be minimal successor or 0 such that $p$ does not
fork over $C\cap A_{i}$, and let $\lambda_{i}'=\lambda_{i_{*}+i}$,
for $i<\mu$. If the lemma is true for $\lambda_{i}'$ (and $A_{i}'=A_{i_{*}+i}$)
instead of $\lambda_{i}$, and $\bar{B}'=\sequence{B_{i}'}{i<\mu}$
witnesses this (so $\left|B_{i}'\right|\leq\lambda_{i}'$, etc.),
then let $B_{i}=A_{i}$ for $i<i_{*}$ and for $i\geq i_{*}$, let
$j<\mu$ be such that $i_{*}+j=i$ and $B_{i}=B_{j}'$.)

Fix an enumeration of $A$ as $\sequence{b_{\alpha}}{\alpha<\kappa}$
such that, for $i<\mu$, $A_{i}$ is enumerated by $\sequence{b_{\alpha}}{\alpha<\lambda_{i}}$.

For each $i<\mu$, let $E_{i}\subseteq\lambda_{i+1}$ be a club such
that:
\begin{itemize}
\item If $\alpha\in E_{i}$ then $\alpha=\omega\alpha$; $\bar{a}_{\alpha}=\sequence{a_{\omega\alpha+k}}{k<\omega}$
is indiscernible over $A_{<\alpha}=\set{b_{\beta}}{\beta<\alpha}$
(which equals $A_{<\omega\alpha}$); $A_{<\alpha}\prec\C_{1}$; if
$\alpha\in E_{i}$ then $\bar{a}_{\beta}\subseteq A_{<\alpha}$ for
all $\beta<\alpha$ and $A_{<\alpha}\supseteq C\cap A_{i}$.
\end{itemize}
Its existence is proved as follows. First note that the set of $\alpha<\lambda_{i+1}$
for which $\alpha=\omega\alpha$ forms a club. Next, since the club
filter on $\lambda_{i+1}$ is $\lambda_{i+1}$-complete, for every
set $D\subseteq A_{i+1}$ of size $<\lambda_{i+1}$, there is a club
$E_{D}\subseteq\lambda_{i+1}$  such that if $\alpha\in E_{D}$ then
$\bar{a}_{\alpha}$ is indiscernible over $D$. Now let $E_{i}'$
be the diagonal intersection $\triangle_{\beta<\lambda_{i+1}}E_{A_{<\beta}}=\set{\alpha<\lambda_{i+1}}{\alpha\in\bigcap_{\beta<\alpha}E_{A_{<\beta}}}$
and let $E_{i}=\set{\alpha\in E_{i}'}{\omega\alpha=\alpha}$. Now,
if $\alpha\in E_{i}$, then $\bar{a}_{\alpha}$ is indiscernible over
$A_{<\beta}$ for all $\beta<\alpha$, and since $\alpha$ is a limit
(as $\omega\alpha=\alpha$), $\bar{a}_{\alpha}$ is indiscernible
over $\bigcup\set{A_{<\beta}}{\beta<\alpha}=A_{<\alpha}$. This takes
care of the requirement that $\bar{a}_{\alpha}$ is indiscernible
over $A_{<\alpha}$. The other requirements follow since $\set{\alpha<\lambda_{i+1}}{A_{<\alpha}\prec\C_{1}}$,
$\set{\alpha<\lambda_{i+1}}{\forall\beta<\alpha\left(\bar{a}_{\beta}\subseteq A_{<\alpha}\right)}$
and $\set{\alpha<\lambda_{i+1}}{C\cap A_{i}\subseteq A_{<\alpha}}$
are clubs. 

Let $E=\bigcup\set{E_{i}}{i<\mu}$. Let $\Gamma\left(x\right)$ be
the set of formulas saying that $p\left(x\right)$ holds and that
for all $\alpha\in E$, $\bar{a}_{\alpha}$ is $L_{1}$-indiscernible
over $A_{<\alpha}\cup x$.
\begin{claim}
It is enough to show that $\Gamma\left(x\right)$ is consistent. 
\end{claim}

\begin{proof}
Let $d\models\Gamma\left(x\right)$. Let $B_{i}'=A_{i}\cup d$ for
all $i<\mu$. Note that for each $i<\mu$ there is a club $E_{i}\subseteq\lambda_{i+1}$
such that if $\alpha\in E_{i}$ then $\bar{a}_{\alpha}$ is indiscernible
over $A_{<\alpha}d$. Now let $c\subseteq B_{i}'$ be any finite set.
Then $c\subseteq A_{<\alpha}d$ for some $\alpha<\lambda_{i+1}$,
so $E'=E_{i}\cap\left[\alpha+1,\lambda_{i}\right)$ is such that for
any $\alpha\in E'$, $\bar{a}_{\alpha}$ is indiscernible over $c$.
Finally, let $B_{i}$ be $\Sk\left(B_{i}'\right)$. Since the indiscernibility
was in $L_{1}$, $\bar{B}=\sequence{B_{i}}{i<\mu}\in\cal C$.
\end{proof}
So fix finite sets $F_{0}\subseteq E$, $v\subseteq\kappa$, and a
finite set of $L_{1}$-formulas $\Delta$, and let $\Gamma_{F_{0},v,\Delta}\left(x\right)$
say that $p\left(x\right)$ holds and for all $\alpha\in F_{0}$,
$\bar{a}_{\alpha}$ is $\Delta$-indiscernible over $\set{b_{\beta}}{\beta\in v\cap\alpha}\cup\left\{ x\right\} $,
and we want to show that $\Gamma_{F_{0},v,\Delta}$ is consistent. 

Let $n=\left|F_{0}\right|$, and write $F_{0}=\set{\alpha_{i}}{i<n}$
where $\alpha_{0}<\ldots<\alpha_{n-1}$. Let $\cal T$ be the tree
of all functions $\left(\left[\omega\right]^{\aleph_{0}}\right)^{\leq n}$
($\left[\omega\right]^{\aleph_{0}}$ is the set of countably infinite
subsets of $\omega$). For every $i<n$ and every infinite $s\subseteq\omega$,
let $f_{i,s}$ be a partial elementary map taking $\bar{a}_{\alpha_{i}}$
to $\bar{a}_{\alpha_{i}}\restriction s$ fixing $A_{<\alpha_{i}}$
(i.e., mapping $a_{\omega\alpha_{i}+k}$ to $a_{\omega\alpha_{i}+k'}$
where $k'$ is the $k$-th element in $s$). Note that $A_{*}=A_{0}\cap C$
is fixed by all the $f_{i,s}$ since $A_{*}\subseteq A_{<\alpha}$
for every $\alpha\in E$ (by the last requirement on $E_{i}$ in the
bullet above). 
\begin{claim}
To show that $\Gamma_{F_{0},v,\Delta}$ is consistent, it is enough
to prove the following: 
\end{claim}

\begin{itemize}
\item [$(\dagger)$]There is an assignment, assigning each $\eta\in\cal T$
of height $i+1\leq n$ an automorphism $\sigma_{\eta}$ of $\C_{1}$
extending $f_{i,\eta\left(i\right)}$ in such a way that, letting
$\tau_{\eta}=\sigma_{\eta\restriction1}\circ\cdots\circ\sigma_{\eta\restriction i+1}$,
$\Theta\left(x\right)=\bigcup\set{\tau_{\eta}\left(p\right)}{\eta\in\cal T}$
is consistent.
\end{itemize}
\begin{proof}
Suppose $d_{*}\models\Theta\left(x\right)$. By Ramsey, there is some
infinite $s_{0}\subseteq\omega$ such that $f_{0,s_{0}}\left(\bar{a}_{\alpha_{0}}\right)=\bar{a}_{\alpha_{0}}\restriction s_{0}$
is $\Delta$-indiscernible over $\set{b_{\beta}}{\beta\in v\cap\alpha_{0}}\cup\left\{ d_{*}\right\} $.
Let $\eta\left(0\right)=s_{0}$. Suppose we chose $\eta\restriction i$
for some $1\leq i<n$. Again by Ramsey there is some infinite $s_{i}\subseteq\omega$
such that, letting $\eta\left(i\right)=s_{i}$ we have that 
\[
\sigma_{\eta\restriction1}\circ\cdots\circ\sigma_{\eta\restriction i+1}\left(\bar{a}_{\alpha_{i}}\right)=\sigma_{\eta\restriction1}\circ\cdots\circ\sigma_{\eta\restriction i}\left(\bar{a}_{\alpha_{i}}\restriction s_{i}\right)=\left[\sigma_{\eta\restriction1}\circ\cdots\circ\sigma_{\eta\restriction i}\left(\bar{a}_{\alpha_{i}}\right)\right]\upharpoonright s_{i}
\]
 is $\Delta$-indiscernible over $\sigma_{\eta\restriction1}\circ\cdots\circ\sigma_{\eta\restriction i}\left(\set{b_{\beta}}{\beta\in v\cap\alpha_{i}}\right)\cup\left\{ d_{*}\right\} $.
This procedure defines $\eta\in\cal T$ of height $n$. Then $\tau_{\eta}^{-1}\left(d_{*}\right)\models p$
and $\bar{a}_{\alpha}$ is $\Delta$-indiscernible over $\left\{ \tau_{\eta}^{-1}\left(d_{*}\right)\right\} \cup\set{b_{\beta}}{\beta\in v\cap\alpha}$
for each $\alpha\in F_{0}$ as we wanted. Indeed,\textcolor{blue}{{}
}$\tau_{\eta}^{-1}\left(d_{*}\right)\models p$ obviously. Also, for
each $i<n$, $\tau_{\eta}^{-1}\left(\sigma_{\eta\restriction1}\circ\cdots\circ\sigma_{\eta\restriction i+1}\left(\bar{a}_{\alpha_{i}}\right)\right)=\sigma_{\eta\restriction n}^{-1}\circ\cdots\circ\sigma_{\eta\restriction1}^{-1}\left(\sigma_{\eta\restriction1}\circ\cdots\circ\sigma_{\eta\restriction i+1}\left(\bar{a}_{\alpha_{i}}\right)\right)=\bar{a}_{\alpha_{i}}$
because for $j>i+1$, $\sigma_{\eta\restriction j}$ fixes $A_{<\alpha_{i+1}}$
which contains $\bar{a}_{\alpha_{i}}$. Hence $\bar{a}_{\alpha_{i}}$
is indiscernible over the union of $\sigma_{\eta}^{-1}\left(d_{*}\right)$
and $\tau_{\eta}^{-1}\circ\sigma_{\eta\restriction1}\circ\cdots\circ\sigma_{\eta\restriction i}\set{b_{\beta}}{\beta\in v\cap\alpha_{i}}$
which is just $\set{b_{\beta}}{\beta\in v\cap\alpha_{i}}$ because
$\sigma_{\eta\restriction j}$ fixes $A_{<\alpha_{i+1}}$ for $j>i$. 
\end{proof}
Fix some enumeration of $\cal T$, $\sequence{\eta_{\varepsilon}}{\varepsilon<\varepsilon_{*}}$
such that $\eta_{0}=\emptyset$, $\varepsilon_{*}$ is a limit and
if $\eta_{\varepsilon}\trianglelefteq\eta_{\zeta}$ then $\varepsilon\leq\zeta$.

Let $\sigma_{0}=\tau_{0}=\id$ and for $0<\varepsilon<\varepsilon_{*}$
define $\sigma_{\varepsilon}=\sigma_{\eta_{\varepsilon}}$ as in $\left(\dagger\right)$
and consequently $\tau_{\varepsilon}=\tau_{\eta_{\varepsilon}}=\sigma_{\eta_{\varepsilon}\restriction1}\circ\cdots\circ\sigma_{\eta_{\varepsilon}\restriction\left|\eta_{\varepsilon}\right|}$
by induction on $\varepsilon$, in such a way that:
\begin{itemize}
\item [$(\dagger_{\epsilon})$] $\Theta_{\varepsilon}\left(x\right)=\bigcup\set{\tau_{\zeta}\left(p\right)}{\zeta<\varepsilon}$
is consistent and does not fork over $A_{*}$.
\end{itemize}
This will obviously suffice in order to prove $\left(\dagger\right)$,
and holds trivially for $\varepsilon=0$ by choice of $A_{*}$. 

Suppose we have already defined $\sigma_{\zeta}$ for all $0\leq\zeta<\varepsilon$,
and let $\eta=\eta_{\varepsilon}$. By assumption on the order, we
already defined $\sigma_{\eta'}$ for the predecessor $\eta'$ of
$\eta$ (recall that $\eta\neq\emptyset$ because $0<\varepsilon$).
Assume $i+1=\left|\eta\right|\leq n$. First, let $\sigma$ be any
automorphism extending $f_{i,\eta\left(i\right)}$ and let $q$ be
a global coheir extending $\tp\left(\tau_{\eta'}\left(\sigma\left(C\right)\right)/N\right)$
where $N=\tau_{\eta'}\circ\sigma\left(\Sk\left(A_{<_{\alpha_{i}}}\bar{a}_{\alpha_{i}}\right)\right)$.
Now let $C'\models q|_{N\tau_{<\varepsilon}\left(C\right)}$ where
$\tau_{<\varepsilon}\left(C\right)=\bigcup\set{\tau_{\zeta}\left(C\right)}{\zeta<\varepsilon}$.
Note that 
\[
\sigma\left(C\right)\sigma\left(\Sk\left(A_{<\alpha_{i}}\bar{a}_{\alpha_{i}}\right)\right)\equiv\tau_{\eta'}\left(\sigma\left(C\right)\right)N\equiv C'N\equiv\tau_{\eta'}^{-1}\left(C'\right)\sigma\left(\Sk\left(A_{<\alpha_{i}}a_{\alpha_{i}}\right)\right).
\]
Let $\sigma'$ be an automorphism mapping $\sigma\left(C\right)$
to $\tau_{\eta'}^{-1}\left(C'\right)$ fixing $\sigma\left(\Sk\left(A_{<_{\alpha_{i}}}\bar{a}_{\alpha_{i}}\right)\right)$.
Now let $\sigma_{\eta}=\sigma'\circ\sigma$. By construction, $\sigma_{\eta}$
 extends $f_{i,\eta\left(i\right)}$. Note that $C'=\tau_{\eta}\left(C\right)\ind_{N}^{u}\tau_{<\varepsilon}\left(C\right)$
(where $\ind^{u}$ means co-heir independence) and that $N=\tau_{\eta}\left(\Sk\left(A_{<_{\alpha_{i}}}\bar{a}_{\alpha_{i}}\right)\right)$
(because $\sigma'\circ\sigma\left(\Sk\left(A_{<_{\alpha_{i}}}\bar{a}_{\alpha_{i}}\right)\right)=\sigma\left(\Sk\left(A_{<_{\alpha_{i}}}\bar{a}_{\alpha_{i}}\right)\right)$). 

Now we check that $\left(\dagger\right)_{\varepsilon+1}$ holds. 

By $\left(\star\right)$ above, we know that $C\ind_{A_{<\alpha_{i}}}\Sk\left(A_{<\alpha_{i}}\bar{a}_{\alpha_{i}}\right)$
since $A_{<\alpha_{i}}\bar{a}_{\alpha_{i}}$ is contained in the appropriate
$A_{i_{0}}$ (i.e., for an $i_{0}<\mu$ for which $\alpha_{i}\in E_{i_{0}}$)
which is a model, and $A_{<\alpha_{i}}$ contains $C\cap A_{i_{0}}$.
Hence, applying $\tau_{\eta}$, we get $\tau_{\eta}\left(C\right)\ind_{\tau_{\eta}\left(A_{<\alpha_{i}}\right)}N$.
By transitivity, we get that $\tau_{\eta}\left(C\right)\ind_{\tau_{\eta}\left(A_{<\alpha_{i}}\right)}\tau_{<\varepsilon}\left(C\right)$.
Note that $\sigma_{\eta}$ fixes $A_{<\alpha_{i}}$, so $\tau_{\eta}\left(A_{<\alpha_{i}}\right)=\tau_{\eta'}\left(A_{<\alpha_{i}}\right)$.

By induction, there is some $d\models\Theta_{\varepsilon}\left(x\right)$
such that $d\ind_{A_{*}}\tau_{<\varepsilon}\left(C\right)$. Let $d_{0}\equiv_{A_{*}\tau_{<\varepsilon}\left(C\right)}d$
be such that $d_{0}\ind_{A_{*}}\tau_{\eta'}\left(A_{<\alpha_{i}}\right)\tau_{<\varepsilon}\left(C\right)$.
Let $d_{1}=\tau_{\eta}\left(\tau_{\eta'}^{-1}\left(d_{0}\right)\right)$.
Since $\eta'$ comes before $\eta$ in the enumeration, we have that
$d_{1}\models\tau_{\eta}\left(\tau_{\eta'}^{-1}\left(\tau_{\eta'}\left(p\right)\right)\right)=\tau_{\eta}\left(p\right)$
and $d_{1}\ind_{A_{*}}\tau_{\eta}\left(A_{<\alpha_{i}}\right)\tau_{\eta}\left(C\right)$
(recall that $A_{*}$ is fixed by $\sigma_{\zeta}$ for all $\zeta\leq\varepsilon$).
Recalling that $\tau_{\eta}\left(A_{<\alpha_{i}}\right)=\tau_{\eta'}\left(A_{<\alpha_{i}}\right)$,
by base monotonicity we have that $d_{0}\ind_{\tau_{\eta'}\left(A_{<\alpha_{i}}\right)}\tau_{<\varepsilon}\left(C\right)$,
$d_{1}\ind_{\tau_{\eta'}\left(A_{<\alpha_{i}}\right)}\tau_{\eta}\left(C\right)$,
$\tau_{\eta}\left(C\right)\ind_{\tau_{\eta'}\left(A_{<\alpha_{i}}\right)}\tau_{<\varepsilon}\left(C\right)$
and $d_{0}\equiv_{\tau_{\eta'}\left(A_{<\alpha_{i}}\right)}d_{1}$
(the last equivalence is because $\tau_{\eta}\circ\tau_{\eta'}^{-1}$
fixes $\tau_{\eta'}\left(A_{<\alpha_{i}}\right)$, because $\tau_{\eta}$
and $\tau_{\eta'}$ agree on $A_{<\alpha_{i}}$, since $\sigma_{\eta}\restriction A_{<\alpha_{i}}=\id$).
By the independence theorem for simple theories (see \cite[Theorem 7.3.11]{TentZiegler}),
we can find some $d_{2}\equiv_{\tau_{\eta'}\left(A_{<\alpha_{i}}\right)\tau_{\eta}\left(C\right)}d_{1}$
(so $d_{2}\models\tau_{\eta}\left(p\right)$), $d_{2}\equiv_{\tau_{\eta'}\left(A_{<\alpha_{i}}\right)\tau_{<\varepsilon}\left(C\right)}d_{0}$
(so $d_{2}\models\Theta_{\varepsilon}\left(x\right)$) and $d_{2}\ind_{\tau_{\eta'}\left(A_{<\alpha_{i}}\right)}\tau_{\leq\varepsilon}\left(C\right)$.
Finally, since $d_{2}\ind_{A_{*}}\tau_{\eta'}\left(A_{<\alpha_{i}}\right)$,
by transitivity we have that $d_{2}\ind_{A_{*}}\bigcup\tau_{\leq\varepsilon}\left(C\right)$.
This finishes the proof of the lemma. 
\end{proof}
Now the proof of Theorem \ref{thm:Main Thm simple theories} continues
precisely as the proof of \cite[Theorem 3.3]{20-Kaplan2015} with
small changes.

First we recall the definition of $\square_{\kappa}$:
\begin{defn}
\label{def:square}(\emph{Jensen's Square principle}, \cite[Page 443]{JechSetTheory})
Let $\kappa$ be an uncountable cardinal; $\square_{\kappa}$ (square-$\kappa$)
is the following condition:

There exists a sequence $\sequence{C_{\alpha}}{\alpha\in\Lim\left(\kappa^{+}\right)}$
such that:
\begin{enumerate}
\item $C_{\alpha}$ is a closed unbounded subset of $\alpha$.
\item If $\beta\in\Lim\left(C_{\alpha}\right)$ then $C_{\beta}=C_{\alpha}\cap\beta$
(where for a set of ordinals $X$, $\Lim\left(X\right)$ is the set
of limit ordinals in $X$).
\item If $\cof\left(\alpha\right)<\kappa$, then $\left|C_{\alpha}\right|<\kappa$. 
\end{enumerate}
\end{defn}

\begin{rem}
\label{rem:Square'}Suppose that $\sequence{C_{\alpha}}{\alpha\in\Lim\left(\kappa^{+}\right)}$
witnesses $\square_{\kappa}$. Let $C_{\alpha}'=\Lim\left(C_{\alpha}\right)$.
Then the following holds for $\alpha\in\Lim\left(\kappa^{+}\right)$. 
\begin{enumerate}
\item If $C_{\alpha}'\neq\emptyset$, then either $\sup\left(C_{\alpha}'\right)=\alpha$,
or $C_{\alpha}'$ has a last element $<\alpha$ in which case $\cof\left(\alpha\right)=\omega$.
If $C_{\alpha}'=\emptyset$ then $\cof\left(\alpha\right)=\omega$
as well. 
\item $C_{\alpha}'\subseteq\Lim\left(\alpha\right)$ and for all $\beta\in C_{\alpha}'$,
$C_{\alpha}'\cap\beta=C_{\beta}'$. 
\item If $\cof\left(\alpha\right)<\kappa$, then $\left|C_{\alpha}'\right|<\kappa$.
\end{enumerate}
Define $\bar{A}\leq_{i}\bar{B}$ for $\bar{A},\bar{B}\in\cal C$ and
$i<\mu$ as $A_{j}\subseteq B_{j}$ for all $i\leq j$ (so $\mathordi{\leq}=\mathordi{\leq_{0}}$
on $\cal C$). Write $\bar{A}\leq_{*}\bar{B}$ for: there is some
$i<\mu$, such that $\bar{A}\leq_{i}\bar{B}$. 
\end{rem}

\begin{proof}[Proof of Theorem \ref{thm:Main Thm simple theories}]
 Let $\sequence{C_{\alpha}}{\alpha\in\Lim\left(\kappa^{+}\right)}$
be a sequence as in Remark \ref{rem:Square'}. Note that $\left|C_{\alpha}\right|<\kappa$
for all $\alpha<\kappa^{+}$ as $\kappa$ is singular. Let $\set{S_{\alpha}}{\alpha<\kappa^{+}}$
be a partition of $\kappa^{+}$ to sets of size $\kappa^{+}$. We
construct a sequence $\sequence{\left(\bar{A}_{\alpha},\bar{p}_{\alpha}\right)}{\alpha<\kappa^{+}}$
such that:
\begin{enumerate}
\item \label{enu:in M}$\bar{A}_{\alpha}=\sequence{A_{\alpha,i}}{i<\mu}\in\cal C$
(recall that $\mu=\cof\left(\kappa\right)$);
\item $\bar{p}_{\alpha}$ is an enumeration $\sequence{p_{\alpha,\beta}}{\beta\in S_{\alpha}\backslash\alpha}$
of all complete types over subsets of $\bigcup_{i}A_{\alpha,i}$ of
size $<\kappa$ (this uses $\kappa^{+}=2^{\kappa}$ and $\left|T\right|\leq\kappa$);
\item \label{enu:star}If $\beta<\alpha$ then $\bar{A}_{\beta}\leqstar\bar{A}_{\alpha}$; 
\item \label{enu:realizing types}If $\alpha\in S_{\gamma}$ and $\gamma\leq\alpha$,
then $\bar{A}_{\alpha+1}$ contains a realization of $p_{\gamma,\alpha}$;
\item \label{enu:limitsquare}If $\alpha$ is a limit ordinal, then for
all $i<\mu$ such that $\left|C_{\alpha}\right|<\lambda_{i}$ and
for all $\beta\in C_{\alpha}$, $\bar{A}_{\beta}\leq_{i}\bar{A}_{\alpha}$. 
\end{enumerate}
The construction is done almost precisely as in \cite[Proof of Theorem 3.3]{20-Kaplan2015},
but we explain. 

Put $A_{0,i}=\Sk\left(I_{i}\right)$. For $\alpha$ successor use
Main Lemma \ref{lem:Main lemma simple}. For $\alpha$ limit, we divide
into two cases. 
\begin{casenv}
\item $\sup\left(C_{\alpha}\right)=\alpha$. Let $i_{0}=\min\set{i<\mu}{\left|C_{\alpha}\right|<\lambda_{i}}$
(which is a successor). For $i<i_{0}$, let $A_{\alpha,i}=A_{0,i}$.
For $i\geq i_{0}$ successor, let $A_{\alpha,i}=\bigcup_{\beta\in C_{\alpha}}A_{\beta,i}$.
Note that $\left|A_{\alpha,i}\right|\leq\lambda_{i}$ for all $i<\mu$.
We have to show that $\bar{A}_{\alpha}$ satisfies (\ref{enu:in M}),
(\ref{enu:star}) and (\ref{enu:limitsquare}). The latter is by construction.

For (\ref{enu:in M}), suppose $s\subseteq A_{\alpha,i}$ is a finite
set where $i_{0}\leq i\in\Succ\left(\kappa\right)$. For every element
$e\in s$, there is some $\beta_{e}\in C_{\alpha}$ such that $e\in A_{\beta_{e},i}$.
Let $\beta=\max\set{\beta_{e}}{e\in s}$. Then $\beta$ is a limit
ordinal and $C_{\alpha}\cap\beta=C_{\beta}$. As $\left|C_{\beta}\right|<\lambda_{i_{0}}$,
it follows by the induction hypothesis that $s\subseteq A_{\beta,i}$.
Hence for some club $E$ of $\lambda_{i+1}$, $\bar{a}_{\alpha}$
is $L_{1}$-indiscernible over $s$ for all $\alpha\in E$. We also
have to check that $A_{\alpha,i}\prec\C_{1}$, but this is immediate
as $A_{\beta,i}\prec\C_{1}$ for all $\beta\in C_{\alpha}$ by induction. 

Lastly, (\ref{enu:star}) is easy by assumption of the case and transitivity
of $\leqstar$.
\item $\sup\left(C_{\alpha}\right)<\alpha$. If $C_{\alpha}=\emptyset$,
let $\gamma=0$ and otherwise let $\gamma=\max C_{\alpha}$ (recall
that it exists). Let $\sequence{\beta_{n}}{n<\omega}$ be a cofinal
increasing sequence in $\alpha$ starting with $\beta_{0}=\gamma$
(which exists since $\cof\left(\alpha\right)=\omega$). For every
$n<\omega$, there is some $i_{n}<\mu$ such that $\bar{A}_{\beta_{n}}\leq_{i_{n}}\bar{A}_{\beta_{n+1}}$.
Without loss of generality assume that $i_{n}<i_{n+1}$ for all $n<\omega$.
Letting $i_{-1}=0$, for all successor $i\geq i_{n-1}$ such that
$i<i_{n}$ define $A_{\alpha,i}=A_{\beta_{n},i}$, and for all successor
$i\geq\sup\set{i_{n}}{n<\omega}$, let $A_{\alpha,i}=\bigcup\set{A_{\beta_{n},i}}{n<\omega}$.
Note that $\bar{A}_{\beta_{n}}\leq_{i_{n-1}}\bar{A}_{\alpha}$ for
all $n<\omega$. This easily satisfies all the requirements. For example
(\ref{enu:limitsquare}): if $C_{\alpha}=\emptyset$, then there is
nothing to check, so assume $C_{\alpha}\neq\emptyset$. Let $i<\mu$
be such that $\left|C_{\alpha}\right|<\lambda_{i}$ and fix some $\beta\in C_{\alpha}$.
Hence $\beta\leq\gamma=\max C_{\alpha}$. Note that $\bar{A}_{\gamma}\leq\bar{A}_{\alpha}$
(so also $\bar{A}_{\gamma}\leq_{i}\bar{A}_{\alpha}$), so we may assume
$\beta<\gamma$. In this case, since $C_{\gamma}=C_{\alpha}\cap\gamma$,
$\beta\in C_{\gamma}$, and since $\left|C_{\gamma}\right|=\left|C_{\alpha}\cap\gamma\right|<\lambda_{i}$,
by induction it follows that $\bar{A}_{\beta}\leq_{i}\bar{A}_{\gamma}\leq\bar{A}_{\alpha}$
so we are done. 
\end{casenv}
Finally, let $M=\bigcup_{\alpha<\kappa^{+},i<\mu}A_{\alpha,i}$. Then
$M$ is a $\kappa$-saturated model of $T$ by (\ref{enu:realizing types}).
However, it is not $\kappa^{+}$-locally saturated because the local
type $\set{\varphi\left(x,a_{j}\right)}{j\in I\mbox{ even}}\cup\set{\neg\varphi\left(x,a_{j}\right)}{j\in I\mbox{ odd}}$
is not realized in $M$. To see this, suppose towards contradiction
that $b$ realizes it. Since $\bar{A}_{\alpha}$ is an increasing
continuous sequence for all $\alpha<\mu^{+}$, there must be some
$\alpha<\mu^{+}$ and $i\in\Succ\left(\kappa\right)$ such that $b\in A_{\alpha,i}$.
But then by (\ref{enu:in M}), for some $\alpha<\lambda_{i}$, $\bar{a}_{\alpha}$
is indiscernible over $b$ --- contradiction. 
\end{proof}

\subsection{Supersimple theories}

From the previous section it follows that if $T$ is countable, unstable
and supersimple ($\kappa\left(T\right)=\aleph_{0}$), and $\kappa$
is singular with cofinality $\omega$ such that $\kappa^{+}=2^{\kappa}$
and $\square_{\kappa}$ holds, then $T$ has PC-exact saturation at
$\kappa$. In this section we will show that this property identifies
supersimplicity among unstable theories: assuming that $T$ is not
supersimple, we will find an expansion $T_{1}$ (of the same size)
such that if $M\models T_{1}$ and $M\restriction L$ is $\kappa$-saturated,
then $M$ is $\kappa^{+}$-saturated. 
\begin{thm}
\label{thm:not supersimple -> singular comp in cof omega}Suppose
that $T$ is unstable and not supersimple, and that $\kappa$ is singular
with $\cof\left(\kappa\right)=\omega$ (in particular uncountable)
and $\kappa\geq\left|T\right|$. Then there is an expansion $T_{1}\supseteq T$
with $\left|T_{1}\right|=\left|T\right|$ and such that if $M\models T_{1}$
is such that $M\restriction L$ is $\kappa$-saturated, then $M\restriction L$
is $\kappa^{+}$-saturated. 
\end{thm}

\begin{proof}
If $T$ is countable we can use essentially the same expansion $T_{1}$
as in Section \ref{subsec:Description-of-the expansion} (adding to
it the function $k$ from below), but since we do not assume that
$T$ is countable we give details. Let $\lambda=\left|T\right|+\aleph_{0}$.
As $T$ is not supersimple, we can use Fact \ref{fact:not supersimple -> 2 incon}
and compactness to find a sequence of formulas $\sequence{\psi_{n}\left(x,y_{n}\right)}{n<\omega}$
and a tree $\sequence{a_{\eta}}{\eta\in\lambda^{<\omega}}$ as there.
In addition, since $T$ is not stable, there is a formula $\varphi\left(x,y\right)$
where $x$ is a singleton and a sequence $\sequence{b_{n},c_{n}}{n<\omega}$
such that $\varphi\left(b_{n},c_{m}\right)$ holds iff $n<m$. 

Let $M_{0}\models T$ be of size $\lambda$ containing $\sequence{a_{\eta}}{\eta\in\lambda^{<\omega}}$
and $\sequence{b_{n}}{n<\omega}$. Without loss of generality we may
assume that the universe of $M_{0}$ is $\lambda\cup\lambda^{<\omega}\cup L$
where $L$ is the set of all formulas from $L$ in a fixed countable
set of variables $\set{v_{i}}{i<\omega}$. We put a predicate $\cal K$
for $\lambda$ and a predicate $\cal T$ for $\lambda^{<\omega}$
on which we add the order $\unlhd$. We put a predicate $\cal N$
on $\omega\subseteq\lambda$ and enrich it with $+,\cdot,<$. Add
two functions $l:\mathcal{T}\to\mathcal{N}$ and $\mathrm{eval}:\mathcal{T}\times\mathcal{N}\to\mathcal{K}$
as in Section \ref{subsec:Description-of-the expansion}: $\mathrm{eval}\left(\eta,n\right)=\eta\left(n\right)$
for $n<l\left(\eta\right)$ and otherwise $\mathrm{eval}\left(\eta,n\right)=0$.
Also add a bijection $e:\cal K\to M_{0}$. Add a predicate $\cal L$
for $L$ and a truth valuation $TV:\cal L\times\cal T\to\left\{ 0,1\right\} $
as in Section \ref{subsec:Description-of-the expansion} (this time
there is no need to add finite subsets of formulas). Add a function
$d:\cal N\to\cal L$ mapping $n$ to $\left\{ \psi_{n}\left(v_{0};v_{1},\ldots,v_{\left|y_{n}\right|}\right)\right\} $
and a function $a:\mathcal{T}\to\mathcal{T}$ such that if $\eta\in\lambda^{<\omega}$
then $a\left(\eta\right)\in\lambda^{\left|y_{l\left(\eta\right)}\right|}$,
$\set{TV\left(d\left(i\right),\left\langle x\right\rangle \concat a\left(\eta|i\right)\right)}{i<l\left(\eta\right)}$
is consistent and such that if $\eta,\nu\in\cal T$ are incomparable
then $\left\{ TV\left(d\left(l\left(\eta\right)\right),\left\langle x\right\rangle \concat a\left(\eta\right)\right),TV\left(d\left(l\left(\nu\right)\right),\left\langle x\right\rangle \concat a\left(\nu\right)\right)\right\} $
is inconsistent. Also add a map $k:\cal{\omega}\to M_{0}$ such that
$k\left(i\right)=b_{i}$. Let $M_{1}=\left(M_{0},\mathcal{N},\cal K,\mathcal{T},\mathcal{L}\right)$
with this additional structure, and set $T_{1}=\mathrm{Th}\left(M_{1}\right)$
and $L_{1}=L\left(T_{1}\right)$.  In models of $T_{1}$, we will
denote by $\omega$ the (interpretations of the) standard natural
numbers, while elements in $\cal N$ which are not from $\omega$
are nonstandard. 
\begin{itemize}
\item [(*)] Note that Lemma \ref{lem:function coding} still holds, with
some minor adjustments to the proof (replacing $\cal N$ by $\cal K$),
giving us a definable function $H:\cal K\times\cal N\to\cal K$ such
that if $M\models T_{1}$ and $M\upharpoonright L$ is $\aleph_{1}$-saturated,
then for any function $f:\omega\to\mathcal{K}^{M}$, there is some
$m_{f}\in\cal K^{M}$ with the property that $H^{M}\left(m_{f},n\right)=f\left(n\right)$
for all $n\in\omega$. 
\end{itemize}
Suppose that $M\models T_{1}$ and $M\restriction L$ is $\kappa$-saturated.
Note that $\kappa\geq\aleph_{1}$ and hence $M\restriction L$ is
$\aleph_{1}$-saturated. Let $p\left(x\right)\in S_{L}\left(A\right)$
be some complete type where $A\subseteq M$ is of size $\left|A\right|=\kappa$.
Since $\cof\left(\kappa\right)=\omega$ we can write $A$ as an increasing
union $A=\bigcup_{i<\omega}A_{i}$ where $\left|A_{i}\right|<\kappa$.
For $i<\omega$, let $c_{i}\models p|_{A_{i}}$. By ({*}), there is
some $m\in\cal K^{M}$ such that $H\left(m,i\right)=e^{-1}\left(c_{i}\right)$
for all $i<\omega$. For every $\varphi\left(x\right)\in p$, there
is $k_{\varphi}<\omega$ such that the set $D_{\varphi}=\set{i\in\cal N^{M}}{\forall j\in\left[k_{\varphi},i\right]M\models\varphi\left(e\left(H\left(m,j\right)\right)\right)}$
contains $\left[k_{\varphi},\omega\right)$. Note that $D_{\varphi}$
is convex, and that it is $A$-definable in $L_{1}$. By overspill,
$D_{\varphi}$ contains some nonstandard element $d_{\varphi}\in\cal N^{M}$.
Let $C=\set{d_{\varphi}}{\varphi\in p}\subseteq\cal N^{M}$. Note
that if $r\in\cal N^{M}$ is nonstandard and $r\leq d_{\varphi}$
then $r\in D_{\varphi}$, so that if $r\leq d_{\varphi}$ for all
$\varphi\in p$ then $e\left(H\left(m,r\right)\right)\models p$ so
that $p$ is realized. 

Since $\left|A\right|\leq\kappa$, $\left|C\right|\leq\kappa$, and
since $\kappa$ is singular, the cofinality of $C$ (going down) is
$<\kappa$. Let $C'\subseteq C$ be a coinitial subset of size $<\kappa$.
To conclude, we show that the set 
\[
\Gamma=\left\{ x\in\cal N\right\} \cup\set{n<x}{n<\omega}\cup\set{x\leq d}{d\in C'}
\]
 is realized in $M$. Recall the choice of $\varphi$ and the function
$k$ above, and consider the set $\Pi=\set{\varphi\left(k\left(i\right),y\right)}{i<\omega}\cup\set{\neg\varphi\left(k\left(d\right),y\right)}{d\in C'}$.
Then $\Pi$ is consistent by choice of $\varphi$, since all elements
in $C'$ are nonstandard. By $\kappa$-saturation $\Pi$ is realized,
say by $f\in M^{\left|y\right|}$. Let $g\in\cal N^{M}$ be minimal
such that $M\models\neg\varphi\left(k\left(g\right),f\right)$. Then
$g$ is nonstandard and $g\leq d$ for all $d\in C'$ (since $\neg\varphi\left(k\left(d\right),f\right)$
holds for all $d\in C'$), so $g\models\Gamma$ and we are done. 
\end{proof}
Combining Theorem \ref{thm:Main Thm simple theories} with Theorem
\ref{thm:not supersimple -> singular comp in cof omega} we get: 
\begin{cor}
\label{cor:supersimple unstable char}Let $T$ be unstable theory.
Suppose that $\kappa$ is singular with cofinality $\omega$ such
that $\left|T\right|<\kappa$, $\kappa^{+}=2^{\kappa}$ and $\square_{\kappa}$
holds. Then, $T$ is supersimple iff $T$ has PC-exact saturation
at $\kappa$. 
\end{cor}

\begin{cor}
\label{cor:countable supersimpmle char}Suppose that $T$ is a countable
theory. Suppose that $\kappa$ is singular with cofinality $\omega$
such that $2^{\aleph_{0}}<\kappa$, $\kappa^{+}=2^{\kappa}$ and $\square_{\kappa}$
holds. Then $T$ is supersimple iff $T$ has PC-exact saturation at
$\kappa$. 
\end{cor}

\begin{proof}
If $T$ is unstable this follows from Corollary \ref{cor:supersimple unstable char}.
If $T$ is stable and supersimple then $T$ is superstable so $T$
is $\kappa$-stable and hence has PC-exact saturation by Coroallry
\ref{cor:PC ES iff stable}. On the other hand, if $T$ has PC-exact
saturation at $\kappa$ then $T$ is $\kappa$-stable. But if $T$
is not superstable then by the stability spectrum theorem \cite[III]{shelah1990classification},
$\kappa^{\aleph_{0}}=\kappa$, contradicting the cofinality assumption. 
\end{proof}

\section{\label{sec:On-local-exact saturation and cofinality omega}On local
exact saturation and cofinality $\omega$}

In this section we will see that for $\kappa$ of cofinality $\omega$,
having local PC-exact saturation defines the class of \emph{supershort}
simple theories: the class of theories for which every local type
does not fork over a finite set. Before doing that, we discuss stable
theories. 

\subsection{Stable theories}

Here we will prove that contrary to the situation with PC-exact saturation
(i.e., to Corollary \ref{cor:PC ES iff stable}), stable theories
always have \uline{local} PC-exact saturation. The proof is similar
to the proof that $\lambda$-stable theories have saturated models
of size $\lambda$ \cite[Theorem III.3.12]{shelah1990classification},
but a bit simpler.
\begin{defn}
\label{def:kappa loc}For any theory $T$, $\kappa_{\loc}\left(T\right)$
is the smallest cardinal $\kappa$ such that any local type $p\in S_{\Delta}\left(A\right)$
(see Definition \ref{def:local types}) does not fork over a set of
size $<\kappa$. If no such cardinal exists, then $\kappa_{\loc}\left(T\right)=\infty$. 
\end{defn}

In other words, the definition is the same as that of $\kappa\left(T\right)$,
where types are replaced with local types. 

For stable theories it is always $\aleph_{0}$:
\begin{prop}
\label{prop:stable -> kappa_loc finite} Suppose that $T$ is stable.
Then $\kappa_{\loc}\left(T\right)=\aleph_{0}$: every local type $p\in S_{\Delta}\left(A\right)$
does not fork over a finite subset $A_{0}\subseteq A$. 
\end{prop}

\begin{proof}
Let $q\in S\left(\C\right)$ be a global non-forking extension of
$p$ over $A$. Then $q$ is definable over $\acl^{\eq}\left(A\right)$.
In particular $q\restriction\Delta$ is definable over $\acl^{\eq}\left(A_{0}\right)$
for some finite subset $A_{0}\subseteq A$, so $q\restriction\Delta$
does not fork over $A_{0}$. 
\end{proof}
\begin{lem}
\label{lem:union of saturated} Suppose that $T$ is a stable $L$-theory.
Suppose that $\sequence{M_{i}}{i<\mu}$ is an increasing continuous
sequence of $\lambda$-locally saturated models. Then $M_{\mu}=\bigcup_{i<\mu}M_{i}$
is also $\lambda$-locally saturated. 
\end{lem}

\begin{proof}
If $\lambda=\aleph_{0}$, this is clear so assume that $\lambda>\aleph_{0}$.
We are given $p\left(x\right)\in S_{\Delta}\left(A\right)$, $A\subseteq M_{\mu}$,
$\left|A\right|<\lambda$, and we want to realize $p$ in $M_{\mu}$
($x$ is any finite tuple of variables). Let $L'$ be a countable
sublanguage of $L$ containing $\Delta$. The models $M_{i}\restriction L'$
are still $\lambda$-locally saturated, so we may assume that $L=L'$
and in particular it is countable. By Proposition \ref{prop:stable -> kappa_loc finite},
$p$ does not fork over a finite set $B\subseteq A$. In particular,
$B\subseteq M_{j}$ for some $j<\mu$. Find a countable model $M'\prec M_{j}$
containing $B$. Let $q$ be a global extension of $p$ which does
not fork over $B$. By stability, $q$ is definable and finitely satisfiable
over $M'$.

By stability, if $\sequence{a_{i}}{i<\omega}$ is any indiscernible
sequence and $\varphi\left(x,y\right)$ is some formula then for any
$b$, either for almost all $i<\omega$ (i.e., all but finitely many)
$\varphi\left(a_{i},b\right)$ holds or for almost all $i<\omega$,
$\neg\varphi\left(a_{i},b\right)$ holds. By compactness, there is
some finite set of formulas $\Delta_{0}$ and $N<\omega$ such that
if $\sequence{a_{i}}{i<2N}$ is any $\Delta_{0}$-indiscernible sequence
then for any (partition of any) formula $\varphi\left(x,y\right)$
from $\Delta$ and any $b$, it cannot be that $\varphi\left(a_{i},b\right)$
for $i<N$ and $\neg\varphi\left(a_{i},b\right)$ for $N\leq i<2N$.

As $M_{j}$ is $\lambda$-locally saturated (and $\lambda>\aleph_{0}$),
we can realize in $M_{j}$ a $\Delta_{0}$-Morley sequence generated
by $q$ over $M'$: $a_{0}\models(q\restriction\Delta_{0})|_{M'}$,
$a_{1}\models(q\restriction\Delta_{0})|_{M'a_{0}}$, etc. It is not
hard to see that in fact the sequence $\sequence{a_{n}}{n<\omega}$
realizes $\left(q^{\left(\omega\right)}|_{M'}\right)\restriction\Delta_{0}$
(the latter is just the restriction of $q^{\left(\omega\right)}|_{M'}$
to the set of formulas from $\Delta_{0}$ in the variables $\sequence{x_{n}}{n<\omega}$
over $M'$ where $\left|x_{n}\right|=\left|x\right|$). For a proof,
see \cite[Claim 4.11]{MR3224981}. In particular, $\sequence{a_{n}}{n<\omega}$
is a $\Delta_{0}$-indiscernible sequence over $M'$. We can continue
and realize in $M_{j}$ a $\Delta_{0}$-Morley sequence $\sequence{a_{i}}{i<\lambda}$
generated by $q$ over $M'$. 

We claim that for some $i<\lambda$, $a_{i}\models p$. Indeed, suppose
not. This means that for every $i<\lambda$, for some $\varphi_{i}\left(x,y\right)$
from $\Delta$ and some $b_{i}\in A$, $\neg\varphi_{i}\left(a_{i},b_{i}\right)$
holds while $\varphi_{i}\left(x,b_{i}\right)\in p$. As $\Delta$
is finite and $\left|A\right|<\lambda$, there are some $\varphi$
and $b\in A$ such that $\neg\varphi\left(a_{i},b\right)$ holds for
all $i\in I_{0}$ where $I_{0}\subseteq\lambda$ is infinite and $\varphi\left(x,b\right)\in p$.
However, we can now realize $a_{0}'\models q\restriction\Delta|_{M'b\cup\set{a_{i}}{i\in I_{0}}}$,
$a_{1}'\models(q\restriction\Delta)|_{M'b\cup\set{a_{i}}{i\in I_{0}}\cup\left\{ a_{0}'\right\} }$,
etc (where $a_{i}'\in\C$). Since $q$ extends $p$, $\varphi\left(a_{i}',b\right)$
holds for all $i<\omega$, so we have a contradiction to the choice
of $\Delta_{0}$. 
\end{proof}

\begin{lem}
\label{lem:existence of locally saturated models}Suppose that $T$
is stable. Then for any $M\models T$ there is an extension $M'\succ M$
such that $M'$ is locally saturated and $\left|M'\right|\leq\left|M\right|+\left|T\right|$. 
\end{lem}

\begin{proof}
We may assume that $\left|T\right|\leq\left|M\right|$. Let $\kappa=\left|M\right|$.

For every finite set of formulas $\Delta$, the number of $\Delta$-types
over $M$ is bounded by $\kappa$ (by stability). Hence there is some
$M_{1}\succ M$ such that $M_{1}$ realizes every local type in $S_{\Delta}\left(M\right)$
and $\left|M_{1}\right|=\kappa$. This allows us to construct a continuous
increasing elementary chain $\sequence{M_{\alpha}}{\alpha<\kappa}$
starting with $M_{0}=M$ with the properties that $\left|M_{\alpha}\right|=\kappa$
and for each $\alpha<\kappa$, $M_{\alpha+1}$ realizes every local
type in $S_{\Delta}\left(M_{\alpha}\right)$. Let $M'=\bigcup_{\alpha<\kappa}M_{\alpha}$.

If $\kappa$ is regular then $M'$ is as requested. 

Otherwise, for every regular $\lambda<\kappa$, $M_{\alpha+\lambda}$
is $\lambda$-locally saturated. Since $M'=\bigcup\set{M_{\alpha+\lambda}}{\alpha<\kappa}$
then by Lemma \ref{lem:union of saturated}, $M'$ is $\lambda$-locally
saturated. Since this is true for every such $\lambda$, $M'$ is
$\kappa$-locally saturated. 
\end{proof}
\begin{thm}
\label{thm:stable local PC}Suppose that $T$ is stable. Then for
every cardinal $\kappa\geq\left|T\right|$, $T$ has local PC-exact
saturation at $\kappa$ . 
\end{thm}

\begin{proof}
Suppose that $T_{1}\supseteq T$ has cardinality $\left|T\right|$.
Without loss of generality assume that $T_{1}$ has Skolem functions.
We construct a sequence of $T$-models $\sequence{M_{i}}{i<\omega}$
and $T_{1}$-models $\sequence{N_{i}}{i<\omega}$ such that:
\begin{itemize}
\item $N_{i}=\Sk\left(M_{i}\right)$; $N_{i}\restriction L\prec M_{i+1}$;
$M_{i}$ is locally saturated and $\left|M_{i}\right|=\kappa$. 
\end{itemize}
For the construction use Lemma \ref{lem:existence of locally saturated models}.
By Lemma \ref{lem:union of saturated}, $M=\bigcup_{i<\omega}M_{i}$
is locally saturated, and by construction it is in $\PC{T_{1}}T$.
It is not $\kappa^{+}$-locally saturated since $\left|M\right|=\kappa$
(so does not realize the local type $\set{x\neq a}{a\in M}$). 
\end{proof}

\subsection{Simple theories}

In this section we will analyze local PC-exact saturation in the context
of simple theories. We start with a positive result:
\begin{thm}
\label{thm:simple local pc}Assume that $T$ is a complete simple
$L$-theory, $T_{1}\supseteq T$ is a theory in $L_{1}\supseteq L$,
$\left|T_{1}\right|\leq\left|T\right|$. Also, assume that $\kappa$
is a singular cardinal such that $\kappa_{\loc}\left(T\right)\leq\mu=\cof\left(\kappa\right)$,
$\left|T\right|<\kappa$, $\kappa^{+}=2^{\kappa}$ and $\square_{\kappa}$
holds. Then $\PC{T_{1}}T$ has local exact saturation at $\kappa$. 
\end{thm}

\begin{proof}
The proof is almost exactly the same as the proof of Theorem \ref{thm:Main Thm simple theories}
(where we also assumed instability). 

By Theorem \ref{thm:stable local PC}, we may assume that $T$ is
unstable, so there is an $L$-formula $\varphi\left(x,y\right)$ with
the independence property. We find an $L_{1}$-indiscernible sequence
$I$ witnessing this and define $\lambda_{i}$, $I_{i}$ for $i<\mu$
as in the proof of Theorem \ref{thm:Main Thm simple theories}. We
also define the class $\cal C$ in exactly the same way. 

The proof of the parallel to Main Lemma \ref{lem:Main lemma simple}
is similar, but the first step is to say that given a local type $p\left(x\right)\in S_{\Delta}\left(C\right)$,
without loss of generality it does not fork over $A_{*}=A_{0}\cap C$,
so we may extend it to a complete type $p'\left(x\right)\in S\left(C\right)$
which also does not fork over $A_{*}$. Then we continue with the
same proof.  

Note also that $x$ may not be a single variable but we never needed
that assumption in the proof of Main Lemma \ref{lem:Main lemma simple}.
\end{proof}
We will now discuss $\kappa_{\loc}\left(T\right)$ (see Definition
\ref{def:kappa loc}), which will lead us to our next result. 

\begin{claim}
For any complete theory $T$ with infinite models, $\kappa_{\loc}\left(T\right)$
(see Definition \ref{def:kappa loc}) can be either $\aleph_{0}$,
$\aleph_{1}$ or $\infty$. In the first two cases $T$ is simple,
and in the last case $T$ is not simple. 
\end{claim}

\begin{proof}
The proof is standard, but we give details. 

If $T$ is not simple, then $T$ has the tree property (see \cite[Definition 7.2.1]{TentZiegler})
as witnessed by some formula $\varphi\left(x,y\right)$ and some $k<\omega$:
there is a sequence $\sequence{a_{s}}{s\in\omega^{<\omega}}$ such
that for every $s\in\omega^{<\omega}$, $\set{\varphi\left(x,a_{s\concat\left\langle i\right\rangle }\right)}{i<\omega}$
is $k$-inconsistent while for any $\eta\in\omega^{\omega}$, $\Gamma_{\eta}=\set{\varphi\left(x,a_{\eta\restriction n}\right)}{n<\omega}$
is consistent. Let $\mu$ be any regular cardinal. By compactness,
we may extend the tree to have width $\lambda=\left(2^{\mu}\right)^{+}$
and height $\mu$ (so that $s$ ranges over $\lambda^{<\mu}$). We correspondingly extend the definition of $\Gamma_{\eta}$ to all $\eta \in \lambda^{\mu}$.  For
$\alpha<\mu$, find an increasing continuous sequence $s_{\alpha}\in\lambda^{\alpha}$ such that $s_{\alpha+1}$ extends $s_{\alpha}$
and $\varphi\left(x,a_{s_{\alpha+1}}\right)$ divides (and even $k$-divides)
over $\set{\varphi\left(x,a_{s_{\beta}}\right)}{\beta\leq\alpha}$
(the construction uses the fact that for infinitely many $i<\lambda$,
$a_{s_{\alpha}\concat\left\langle i\right\rangle }$ will have the
same type over $a_{s_{<\alpha}}$). Letting $\eta=\bigcup_{\alpha<\mu}s_{\alpha}$,
any complete $\varphi$-type extending $\Gamma_{\eta}$ over $\set{a_{s_{\alpha}}}{\alpha<\mu}$
divides over any subset of size $<\mu$ its domain. Since $\mu$ was
arbitrary, this show that $\kappa_{\loc}\left(T\right)=\infty$. 

Now, if $\kappa_{\loc}\left(T\right)>\aleph_{1}$, then there is a
local type $p\in S_{\Delta}\left(A\right)$ which forks over any countable
subset of $A$. Let $L'$ be a countable subset of the language $L$
of $T$ containing all the symbols appearing in $\Delta$. Then any
completion $q\in S_{L'}\left(A\right)$ of $p$ forks over any countable
subset of $A$, so $T\restriction L'$ does not satisfy local character
for non-forking, so it is not simple, and so is $T$, thus $\kappa_{\loc}\left(T\right)=\infty$. 

Finally, $\kappa_{\loc}\left(T\right)$ cannot be any $n<\omega$,
since given $a_{0},\ldots,a_{n-1}$ with $a_{i}\notin\acl\left(a_{\neq i}\right)$
(e.g., $a_{i}$ come from an infinite indiscernible sequence), $\tp_{=}\left(a_{0},\ldots,a_{n-1}/a_{0},\ldots,a_{n-1}\right)$
divides over any proper subset of $\set{a_{i}}{i<n}$.
\end{proof}
\begin{defn}
\cite[Definition 8]{MR1905165} A theory is called \emph{supershort}
if $\kappa_{\loc}\left(T\right)=\aleph_{0}$. 
\end{defn}

\begin{rem}
\label{rem:supershort discussion}This is not the precise definition
given in \cite{MR1905165} which is given in terms of dividing chains,
but it is equivalent to it: given an infinite dividing chain of conjunctions
of a single formula $\varphi\left(x,y\right)$ as in the definition
there, the partial $\varphi$-type containing them divides over any
finite subset of its domain. On the other hand, if $\kappa_{\loc}\left(T\right)>\aleph_{0}$
and $p\in S_{\Delta}\left(A\right)$ witnesses this (i.e., divides
over every finite $A_{0}\subseteq A$) for some finite $\Delta$,
then by coding finitely many formulas as one formula (see \cite[Proof of Theorem II.2.12(1)]{shelah1990classification}),
we can recover a dividing chain as in the definition in \cite{MR1905165}. 
\end{rem}

Recall that a theory $T$ is \emph{low} if whenever $\varphi\left(x,y\right)$
is a formula then there is some $n<\omega$ such that if $\sequence{a_{i}}{i<\omega}$
is an indiscernible sequence such that $\set{\varphi\left(x,a_{i}\right)}{i<\omega}$
is inconsistent, then it is already $n$-inconsistent. This is not
the original definition from \cite{MR1777789,MR1812164}, which is
given using local ranks, but it is equivalent to it when $T$ is simple,
see \cite[Proposition 18.19]{MR2814891}. 

The following proposition gives a simple criterion for supershortness. 
\begin{prop}
Suppose that $T$ is a simple theory such that if $\Delta$ is a finite
set of formulas and $p\left(x\right)\in S_{\Delta}\left(A\right)$
then there is a finite set of formulas $\Delta'$ such that whenever
$p$ divides over $A_{0}\subseteq A$, there is some formula $\varphi\left(x,y\right)$
from $\Delta'$ and some $a\in\C$ such that $p\vdash\varphi\left(x,a\right)$
and $\varphi\left(x,a\right)$ divides over $A_{0}$. Then if $T$
is low then $\kappa_{\loc}\left(T\right)=\aleph_{0}$.
\end{prop}

\begin{proof}
If $\kappa_{\loc}\left(T\right)>\aleph_{0}$ then there is a local
type $p\left(x\right)\in S_{\Delta}\left(A\right)$ (for $\Delta$
finite) such that $p$ divides over any finite subset of $A$, in
particular $A$ is infinite. Let $\Delta'$ be as above. Thus we can
construct an increasing chain $\sequence{A_{i}}{i<\omega}$ such that
$A_{i}\subseteq A$ are finite and $p|_{A_{i+1}}$ divides over $A_{i}$.
As $\Delta'$ is finite we can find a single formula $\varphi\left(x,y\right)\in\Delta'$
and $a_{i}\in\C$ such that $\varphi\left(x,a_{i}\right)$ divides
over $A_{i}$ and $p|_{A_{i+1}}\vdash\varphi\left(x,a_{i}\right)$.
If $J_{i}$ is an indiscernible sequence over $A_{i}$ witnessing
that $\varphi\left(x,a_{i}\right)$ divides over $A_{i}$, by Ramsey
and compactness and applying an automorphism we can assume that $J_{i}$
is indiscernible over $a_{<i}$ (perhaps changing $a_{<i}$). By compactness
we can assume that $\set{\varphi\left(x,a_{i}\right)}{i<\omega}$
is consistent and $\varphi\left(x,a_{i}\right)$ divides over $a_{<i}$
(we need compactness since we changed $a_{<i}$ in every stage). As
$T$ is low, there is some $k$ such that $\varphi\left(x,a_{i}\right)$
$k$-divides over $a_{<i}$. From this we can recover the tree property.
Alternatively, this also follows from \cite[Proposition 18.19 (5)]{MR2814891}. 
\end{proof}
However, in general there is no connection between being low and being
supershort. We found the following table useful. 
\begin{center}
\begin{tabular}{|c|c|c|}
\hline 
 & Supershort & Not supershort\tabularnewline
\hline 
\hline 
Low & Any stable theory & \cite[Section 4]{MR1905165}\tabularnewline
\hline 
Not low & \cite{MR1776222} (even supersimple) & \cite[Section 5]{MR1696842}\tabularnewline
\hline 
\end{tabular}
\par\end{center}

Our goal is to show that when $\cof\left(\kappa\right)=\aleph_{0}$,
$T$ has local PC-exact saturation at $\kappa$ if and only if $T$
is supershort. 
\begin{prop}
\label{prop:2-inconsistency tree}If $\kappa_{\loc}\left(T\right)>\aleph_{0}$
then there is a formula $\varphi\left(x,y\right)$ and a sequence
of formulas $\sequence{\psi_{n}\left(x,y_{n}\right)}{n<\omega}$ (where
$x$ is a finite tuple of variables and the $y_{n}$'s are tuples
of variables of varying lengths) and a sequence $\sequence{a_{\eta}}{\eta\in\omega^{<\omega}}$
such that:
\begin{itemize}
\item Each formula $\psi_{n}$ has the form $\bigwedge_{j<l}\varphi\left(x,y_{j}\right)$
for some $l$; $a_{\eta}$ is an $\left|y_{\left|\eta\right|}\right|$-tuple;
For $\sigma\in\omega^{\omega}$, $\set{\psi_{n}\left(x,a_{\sigma\restriction n}\right)}{n<\omega}$
is consistent; For every $\eta\in\omega^{n},\nu\in\omega^{m}$ such
that $\eta\perp\nu$, $\left\{ \psi_{n}\left(x,a_{\eta}\right),\psi_{m}\left(x,a_{\nu}\right)\right\} $
is 2-inconsistent. 
\end{itemize}
\end{prop}

\begin{proof}
This is a local version of Fact \ref{fact:not supersimple -> 2 incon}.
Since $\kappa_{\loc}\left(T\right)>\aleph_{0}$, there is a finite
set $\Delta$ and a type $p\left(x\right)\in S_{\Delta}\left(A\right)$
(for some infinite set $A$) that forks over every finite subset $A_{0}\subseteq A$.
By Remark \ref{rem:supershort discussion}, there is a formula $\varphi\left(x,y\right)$
and a sequence of formulas $\sequence{\psi_{n}\left(x,y_{n}\right)}{n<\omega}$
where each $\psi_{n}$ is a conjunction of the form $\bigwedge_{i<k_{n}}\varphi\left(x,y_{i}\right)$,
and a sequence $\sequence{a_{n}}{n<\omega}$ such that $\set{\psi_{n}\left(x,a_{n}\right)}{n<\omega}$
is consistent and $\psi_{n}\left(x,a_{n}\right)$ divides over $a_{<n}$.
Thus there are $\sequence{k_{n}<\omega}{n<\omega}$ and a tree $\sequence{a_{\eta}}{\eta\in\omega^{<\omega}}$
such that for every $\sigma\in\omega^{\omega}$, $\set{\psi_{n}\left(x,a_{\sigma\restriction n}\right)}{n<\omega}$
is consistent and such that for every $n<\omega$ and $\eta\in\omega^{n}$,
$\set{\psi_{n+1}\left(x,a_{\eta\concat\left\langle i\right\rangle }\right)}{i<\omega}$
is $k_{n}$-inconsistent. By applying the same proof as in \cite[Proposition 3.5]{ArtemNick}
to this tree, we are done. 
\end{proof}
\begin{thm}
Let $T$ be any complete theory. Suppose that $\kappa$ is singular
with cofinality $\omega$ such that $\left|T\right|<\kappa$, $\kappa^{+}=2^{\kappa}$
and $\square_{\kappa}$ holds. Then, $T$ is supershort iff $T$ has
local PC-exact saturation at $\kappa$. 
\end{thm}

\begin{proof}
Right to left follows from Theorem \ref{thm:simple local pc}. For
the other direction, use the same proof as in Theorem \ref{thm:not supersimple -> singular comp in cof omega},
noting that the proof goes through, because the only use of actual
$\kappa$-saturation as opposed to local $\kappa$-saturation was
the use of ({*}) (i.e., the use of the suitable version of Lemma \ref{lem:function coding}).
Here, all the formulas $\psi_{n}$ are conjunctions of instances of
$\varphi$ so the types 
\[
\set{\psi_{i+1}\left(x,e\left(a\left(\eta_{i}\right)\left(0\right)\right),\ldots,e\left(a\left(\eta_{i}\right)\left(\left|y_{i+1}\right|-1\right)\right)\right)}{i<\omega}
\]
 (using the notation from the proof of Lemma \ref{lem:function coding})
are still consistent. Of course, since the $x$ may now be a tuple
of length $>1$, the function $H$ has domain $\cal K^{\left|x\right|}\times\cal N$
(where $x$ is from $\psi_{n}\left(x,y_{n}\right)$). 

One more difference is that now the type $p\left(z\right)$ we wish
to realize is in possibly more than one variable. However this is
easy to overcome by taking a tuple of ``codes'' for the function
$n\mapsto c_{n}$. 
\end{proof}

\section{\label{sec:Final-thoughts}Final thoughts and questions}

\subsection{NSOP$_{1}$}

We would like to extend Theorem \ref{sec:PC-exact-saturation-for simple }
to NSOP$_{1}$-theories, but we do not even know the situation with
elementary classes (i.e., not PC-exact saturation, just exact saturation).
The approach of mimicking the proof or the proof of \cite[Theorem 3.3]{20-Kaplan2015}
using Kim-independence and all its properties (see \cite{kaplan2019transitivity,kaplan2017kim,24-Kaplan2017}) does not seem to work without new ideas.
Both proofs use base monotonicity and hence are not applicable. 
\begin{question}
\label{que:NSOP1 extension}Is Theorem \ref{sec:PC-exact-saturation-for simple }
true for NSOP$_{1}$-theories?
\end{question}

Note that \cite[Theorem 9.30]{MR3666452} states that if $T$ has
SOP$_{2}$ then it has \emph{PC-singular compactness} (the negation
of PC-exact saturation): for some $T_{1}$ containing $T$ of cardinality
$\leq\left|T\right|$ and every singular $\kappa>\left|T\right|,$
if $M\in PC\left(T,T_{1}\right)$ is $\kappa$-saturated then it is
$\kappa^{+}$-saturated. Thus, a positive answer to Question \ref{que:NSOP1 extension}
will help to ``close the gap''. 

\subsection{NIP}

In \cite[Theorem 4.10]{20-Kaplan2015} it is proved that if $T$ is
NIP, and $\left|T\right|<\kappa$ is singular such that $2^{\kappa}=\kappa^{+}$,
then $T$ has exact saturation at $\kappa$ iff $T$ is not distal.
While the situation for PC-exact saturation seems less clear, one
can ask about local exact saturation (without PC). The proof of the
direction that if $T$ is distal then $T$ does not have exact saturation
at $\kappa$ goes through in the local case: if $T$ is distal, $\left|T\right|<\kappa$
is singular, then every $\kappa$-locally saturated model is $\kappa^{+}$-locally
saturated. This is Proposition \cite[Proposition 4.12]{20-Kaplan2015}.
The proof has to be adjusted. Following the notation there, we elaborate
a bit. Given a finite set $\Delta$ and a type $p\left(x\right)\in S_{\Delta}\left(A\right)$,
let $p'$ be an extension of $p$ to $S_{\Delta'}\left(A\right)$
where $\Delta'=\Delta\cup\set{\theta^{\varphi}}{\varphi\in\Delta}$
(we also consider all possible partitions of formulas in $\Delta$).
We let $b_{i}\models p'|_{A_{i}}$ for $i<\mu$ and find $d_{i}^{\varphi}$
as there for any $\varphi\in\Delta$. Letting $q_{i}=\set{\theta^{\varphi}\left(x,d_{i}^{\varphi}\right)}{\varphi\in\Delta}$
(so a finite set), the proof of \cite[Proposition 4.13]{20-Kaplan2015}
goes through because $p'$ is a complete $\Delta'$-type. We then
find $e_{i}$ realizing the $\Delta''$-type of the finite tuple $d_{i}=\sequence{d_{i}^{\varphi}}{\varphi\in\Delta}$
over $A_{i}\cup\set{b_{i}}{i<\mu}$ where $\Delta''$ contains $\Delta'$
and the formulas $(\forall x)(\theta^{\varphi}\left(x,z\right)\to\varphi\left(x,y\right))$
and $(\forall x)(\theta^{\varphi}\left(x,z\right)\to\neg\varphi\left(x,y\right))$
for $\varphi\in\Delta$. The rest goes through. 

However, the other direction, namely that if $T$ is not distal then
$T$ has local exact saturation at $\kappa$ for $\kappa$ as above
seems less clear. The main issue is that the model constructed omits
a type of an element over an indiscernible sequence, and this type
is not local. For example if $I$ is an indiscernible set, then the
type omitted is that of a new element in the sequence. 
\begin{question}
Which NIP theories have local exact saturation at singular cardinals
as above?
\end{question}

\bibliographystyle{alpha}
\bibliography{ms}

\begin{thebibliography}{KRS21}

\bibitem[Bue99]{MR1777789}
Steven Buechler.
\newblock Lascar strong types in some simple theories.
\newblock {\em J. Symbolic Logic}, 64(2):817--824, 1999.

\bibitem[Cas99]{MR1696842}
Enrique Casanovas.
\newblock The number of types in simple theories.
\newblock {\em Ann. Pure Appl. Logic}, 98(1-3):69--86, 1999.

\bibitem[Cas11]{MR2814891}
Enrique Casanovas.
\newblock {\em Simple theories and hyperimaginaries}, volume~39 of {\em Lecture
  Notes in Logic}.
\newblock Association for Symbolic Logic, Chicago, IL; Cambridge University
  Press, Cambridge, 2011.

\bibitem[CK98]{MR1776222}
Enrique Casanovas and Byunghan Kim.
\newblock A supersimple nonlow theory.
\newblock {\em Notre Dame J. Formal Logic}, 39(4):507--518, 1998.

\bibitem[CR16]{ArtemNick}
Artem Chernikov and Nicholas Ramsey.
\newblock On model-theoretic tree properties.
\newblock {\em Journal of Mathematical Logic}, page 1650009, 2016.

\bibitem[CW02]{MR1905165}
Enrique Casanovas and Frank~O. Wagner.
\newblock Local supersimplicity and related concepts.
\newblock {\em J. Symbolic Logic}, 67(2):744--758, 2002.

\bibitem[Jec03]{JechSetTheory}
Thomas Jech.
\newblock {\em Set theory}.
\newblock Springer Monographs in Mathematics. Springer-Verlag, Berlin, 2003.
\newblock The third millennium edition, revised and expanded.

\bibitem[KR20]{kaplan2017kim}
Itay Kaplan and Nicholas Ramsey.
\newblock On {K}im-independence.
\newblock {\em J. Eur. Math. Soc. (JEMS)}, 22(5):1423--1474, 2020.

\bibitem[KR21]{kaplan2019transitivity}
Itay Kaplan and Nicholas Ramsey.
\newblock Transitivity of {K}im-independence.
\newblock {\em Adv. Math.}, 379:Paper No. 107573, 29, 2021.

\bibitem[KRS19]{24-Kaplan2017}
Itay Kaplan, Nicholas Ramsey, and Saharon Shelah.
\newblock Local character of {K}im-independence.
\newblock {\em Proc. Amer. Math. Soc.}, 147(4):1719--1732, 2019.

\bibitem[KRS21]{kaplan2020criteria}
Itay Kaplan, Nicholas Ramsey, and Saharon Shelah.
\newblock Criteria for exact saturation and singular compactness.
\newblock {\em Ann. Pure Appl. Logic}, 172(9):Paper No. 102992, 28, 2021.

\bibitem[KS14]{MR3224981}
Itay Kaplan and Saharon Shelah.
\newblock Examples in dependent theories.
\newblock {\em J. Symb. Log.}, 79(2):585--619, 2014.

\bibitem[KSS17]{20-Kaplan2015}
Itay Kaplan, Saharon Shelah, and Pierre Simon.
\newblock Exact saturation in simple and {NIP} theories.
\newblock {\em J. Math. Log.}, 17(1):1750001, 18, 2017.

\bibitem[MS17]{MR3666452}
Maryanthe Malliaris and Saharon Shelah.
\newblock Model-theoretic applications of cofinality spectrum problems.
\newblock {\em Israel J. Math.}, 220(2):947--1014, 2017.

\bibitem[Poi00]{MR1757487}
Bruno Poizat.
\newblock {\em A course in model theory}.
\newblock Universitext. Springer-Verlag, New York, 2000.
\newblock An introduction to contemporary mathematical logic, Translated from
  the French by Moses Klein and revised by the author.

\bibitem[Sha00]{MR1812164}
Ziv Shami.
\newblock Definability in low simple theories.
\newblock {\em J. Symbolic Logic}, 65(4):1481--1490, 2000.

\bibitem[She90]{shelah1990classification}
Saharon Shelah.
\newblock {\em Classification theory: and the number of non-isomorphic models}.
\newblock Elsevier, 1990.

\bibitem[Sim13]{MR3001548}
Pierre Simon.
\newblock Distal and non-distal {NIP} theories.
\newblock {\em Ann. Pure Appl. Logic}, 164(3):294--318, 2013.

\bibitem[TZ12]{TentZiegler}
Katrin Tent and Martin Ziegler.
\newblock {\em A course in model theory}, volume~40 of {\em Lecture Notes in
  Logic}.
\newblock Association for Symbolic Logic, La Jolla, CA; Cambridge University
  Press, Cambridge, 2012.

\end{thebibliography}

\end{document}